\input{ifpdf.sty}
\pdffalse

\makeatletter

\newdimen\paperwidth
\newdimen\paperheight

\def\papersize#1#2{\let\p@persize\relax\paperwidth#1\paperheight#2}

\def\Afour{\papersize{210truemm}{297truemm}}

\let\p@persize\Afour

\let\onesidestyle\@twosidefalse
\let\twosidestyle\@twosidetrue

\def\margins{\@ifnextchar[{\@margins}{\@margins[\z@]}}

\def\@margins[#1]#2#3{
  \p@persize\dimen0 #3\dimen0 .5\dimen0\normalsize%
  \oddsidemargin-1truein\advance\oddsidemargin#2%
  \evensidemargin-1truein\advance\evensidemargin#2%
  \topmargin-1truein\advance\topmargin\dimen0\headsep\dimen0\footskip\dimen0%
  \textwidth\paperwidth\advance\textwidth-#2\advance\textwidth-#2%
  \textheight\paperheight\advance\textheight-#3\advance\textheight-#3%
  \headheight\baselineskip\advance\topmargin-.5\baselineskip%
  \advance\headsep-.5\baselineskip%
  \footheight\baselineskip
  \advance\textwidth-#1\advance\oddsidemargin#1
  \if@twoside\def\@themargin%
    {\ifodd\count\z@\oddsidemargin\else\evensidemargin\fi}\fi}

\def\headlinesep#1{\advance\topmargin\headsep\advance\topmargin -#1
  \advance\topmargin.5\baselineskip\headsep #1\advance\headsep-.5\baselineskip}

\def\headline{\if@twoside\let\n@xt\h@dlin@\else\let\n@xt\h@@dlin@\fi\n@xt}
  
\def\h@dlin@#1#2{%
  \def\@oddhead{%
    {{\leftskip\z@\rightskip\z@\noindent\normalsize#1}}}
  \def\@evenhead{%
    {{\leftskip\z@\rightskip\z@\noindent\normalsize#2}}}}

\def\h@@dlin@#1{%
  \def\@oddhead{{{\leftskip\z@\rightskip\z@\noindent\normalsize#1}}}}

\def\footline{\if@twoside\let\n@xt\f@tlin@\else\let\n@xt\f@@tlin@\fi\n@xt}

\def\f@tlin@#1#2{%
  \def\@oddfoot{%
    {{\leftskip\z@\rightskip\z@\noindent\normalsize#1}}}
  \def\@evenfoot{%
    {{\leftskip\z@\rightskip\z@\noindent\normalsize#2}}}}

\def\f@@tlin@#1{%
  \def\@oddfoot{{{\leftskip\z@\rightskip\z@\noindent\normalsize#1}}}}

\def\normalpage{\global\@specialpagefalse}

\makeatother

\makeatletter

\def\ft{\@ifnextchar[{\ft@s}{\ft@}}
\def\ft@{\ft@@@s[\f@size]}
\def\ft@s[{\@ifnextchar{a}{\ft@sz[}{\ft@@s[}}
\def\ft@@s[{\@ifnextchar{s}{\ft@sz[}{\ft@@@s[}}
\def\ft@@@s[#1]{\ft@sz[at #1pt]}
\def\ft@sz[#1]#2{\font\fonttemp=#2 #1\fonttemp\ignorespaces}

\def\pbf{\fontseries\bfdefault\selectfont\boldmath}

\makeatother

\makeatletter

\def\@@bold{bold}

\def\widebar{\ifx\math@version\@@bold
  \let\@widebar\@@@widebar\else\let\@widebar\@@widebar\fi\@widebar}

\def\@@widebar#1{\text{\setbox15\hbox{$#1$}%
  \dimen15 0.45\wd15\advance\dimen15 0.15\ht15%
  \dimen16\ht15\advance\dimen16 0.00em\advance\dimen16 0.3ex%
  \dimen17 0.65\wd15\advance\dimen17 0.05\ht15\advance\dimen17 0.1ex%
  \dimen18 0.035em\advance\dimen18 0.00ex
  \put[\dimen15,\dimen16][c]{\vrule depth 0pt height \dimen18 width \dimen17}}#1}

\def\@@@widebar#1{\text{\setbox15\hbox{$#1$}%
  \dimen15 0.45\wd15\advance\dimen15 0.15\ht15%
  \dimen16\ht15\advance\dimen16 0.00em\advance\dimen16 0.26ex%
  \dimen17 0.65\wd15\advance\dimen17 0.05\ht15\advance\dimen17 0.1ex%
  \dimen18 0.05em\advance\dimen18 0.00ex
  \put[\dimen15,\dimen16][c]{\vrule depth 0pt height \dimen18 width \dimen17}}#1}

\makeatother

\def\smallsquare{\raise-.065em\hbox{$\Box$}}
\def\smallblacksquare{%
  \kern.3ex\vrule depth-.03ex height1.27ex width1.15ex \kern-1.45ex \smallsquare}

\def\smallcircc{\mathop{\mkern3.5mu\text{\raise.58ex\hbox{\ft{lcircle10}a}}}}
\def\varemptyset{{\text{\raise.21ex\hbox{$\not$}}\mkern.15mu\mathrm{O}\mkern.15mu}}

  \let\epsilon\varepsilon
      \let\theta\vartheta
          \let\phi\varphi
   \let\emptyset\varemptyset

\documentstyle[12pt]{article}

\makeatletter

\let\Larg@\Large
\let\hug@\huge

\def\usepackage#1{\input{#1.sty}}

\input{geom.sty}
\input{option_keys.sty}

\let\input\@input

\def\r@adlabel#1#2{\global\@namedef{#1@\the\@key}{#2}}

\let\Large\Larg@
\let\huge\hug@

\def\smallskip{\vskip\smallskipamount}
\def\medskip{\vskip\medskipamount}
\def\bigskip{\vskip\bigskipamount}

\def\mytrivlist{\parsep\parskip\@nmbrlistfalse
  \my@trivlist \labelwidth\z@ \leftmargin\z@
  \itemindent\z@ \def\makelabel##1{##1}}

\def\my@trivlist{\global\@newlisttrue \@outerparskip\parskip}

\def\end#1{\csname end#1\endcsname\@checkend{#1}%
  \expandafter\endgroup\if@endpe\@doendpe\fi
  \if@ignore \global\@ignorefalse \ignorespaces\fi}

\input{epsf.sty}
\usepackage{graphicx}

\def\put{\@ifnextchar[{\@put}{\@@rput[\z@,\z@][r]}}
\def\@put[#1]{\@ifnextchar[{\@@put[#1]}{\@@@@@put[#1]}}
\def\@@put[#1][{\@ifnextchar{l}{\@@lput[#1][}{\@@@put[#1][}}
\def\@@@put[#1][{\@ifnextchar{c}{\@@cput[#1][}{\@@@@put[#1][}}
\def\@@@@put[#1][{\@ifnextchar{r}{\@@rput[#1][}{\relax}}
\def\@@@@@put[{\@ifnextchar{l}{\@@lput[\z@,\z@][}{\@@@@@@put[}}
\def\@@@@@@put[{\@ifnextchar{c}{\@@cput[\z@,\z@][}{\@@@@@@@put[}}
\def\@@@@@@@put[{\@ifnextchar{r}{\@@rput[\z@,\z@][}{\@@@@@@@@put[}}
\def\@@@@@@@@put[#1]{\@@rput[#1][r]}

\let\hm@d@\leavevmode

\long\def\@@lput[#1,#2][l]#3{\setbox0\hbox{#3}\hm@d@\raise#2\hbox to\z@{\dimen0 #1%
  \advance\dimen0-\wd0\kern\dimen0\dp0\z@\ht0\z@\wd0\z@\box0\hss}\ignorespaces}
\long\def\@@cput[#1,#2][c]#3{\setbox0\hbox{#3}\hm@d@\raise#2\hbox to\z@{\dimen0 #1%
  \advance\dimen0-.5\wd0\kern\dimen0\dp0\z@\ht0\z@\wd0\z@\box0\hss}\ignorespaces}
\long\def\@@rput[#1,#2][r]#3{\setbox0\hbox{\kern#1\raise#2\hbox{#3}}%
  \dp0\z@\ht0\z@\wd0\z@\hm@d@\box0\ignorespaces}

\def\flbox{\@ifnextchar[{\@flbox}{\@@rflbox[\z@,\z@][r]}}
\def\@flbox[#1]{\@ifnextchar[{\@@flbox[#1]}{\@@@@@flbox[#1]}}
\def\@@flbox[#1][{\@ifnextchar{l}{\@@lflbox[#1][}{\@@@flbox[#1][}}
\def\@@@flbox[#1][{\@ifnextchar{c}{\@@cflbox[#1][}{\@@@@flbox[#1][}}
\def\@@@@flbox[#1][{\@ifnextchar{r}{\@@rflbox[#1][}{\relax}}
\def\@@@@@flbox[{\@ifnextchar{l}{\@@lflbox[\z@,\z@][}{\@@@@@@flbox[}}
\def\@@@@@@flbox[{\@ifnextchar{c}{\@@cflbox[\z@,\z@][}{\@@@@@@@flbox[}}
\def\@@@@@@@flbox[{\@ifnextchar{r}{\@@rflbox[\z@,\z@][}{\@@@@@@@@flbox[}}
\def\@@@@@@@@flbox[#1]{\@@rflbox[#1][r]}
\long\def\@@lflbox[#1,#2][l]#3{\@@lput[#1,#2][l]{%
  \vtop{\leftskip\z@\parindent\z@\raggedleft\hm@d@#3}}}
\long\def\@@cflbox[#1,#2][c]#3{\@@cput[#1,#2][c]{%
  \vtop{\leftskip\z@\parindent\z@\raggedcenter\hm@d@#3}}}
\long\def\@@rflbox[#1,#2][r]#3{\@@rput[#1,#2][r]{%
  \vtop{\leftskip\z@\parindent\z@\raggedright\hm@d@#3}}}

\def\maketitle{\par
 \begingroup
 \def\thefootnote{\fnsymbol{footnote}}
 \def\@makefnmark{\hbox 
 to 0pt{$^{\@thefnmark}$\hss}} 
 \if@twocolumn 
 \twocolumn[\@maketitle] 
 \else 
 \global\@topnum\z@ \@maketitle \fi\thispagestyle{plain}\@thanks
 \endgroup
 \setcounter{footnote}{0}
 \let\maketitle\relax
 \let\@maketitle\relax
 \gdef\@thanks{}\gdef\@author{}\gdef\@title{}\let\thanks\relax}

\def\@maketitle{ 
 \null
 \vskip 2em \begin{center}
 {\LARGE \@title \par} \vskip 1.5em {\large \lineskip .5em
\begin{tabular}[t]{c}\@author 
 \end{tabular}\par} 
 \vskip 1em {\large \@date} \end{center}
 \par
 \vskip 1.5em}
 
\def\partbeforeskip#1{\def\p@rtbeforeskip{#1}}
\def\partstyle#1{\def\p@rtstyl@{#1}}
\def\partdot#1{\def\partd@t{#1}}
\def\partafterskip#1{\def\p@rtafterskip{#1}}
\def\partintrostyle#1{\def\partintr@styl@{#1}}
\def\partintrodot#1{\def\partintr@dot{#1}}
\long\def\partintrosep#1{\long\def\partintr@sep{#1}}
\def\partnewpagetrue{\def\p@rtnewp@ge{\newpage}}
\def\partnewpagefalse{\long\def\p@rtnewp@ge{\par}}

\partbeforeskip{4ex}
\partstyle{\centering\Large\bf}
\partdot{}
\partafterskip{3ex}
\partintrostyle{\large}
\partintrodot{}
\partintrosep{\par}
\partnewpagefalse

\def\partname{Part}
\def\part{\p@rtnewp@ge\addvspace\p@rtbeforeskip\@afterindentfalse\secdef\@part\@spart}

\def\@part[#1]#2{\ifnum \c@secnumdepth >-1\relax  
        \refstepcounter{part}                     
        \def\@tempa{\addcontentsline{toc}{part}}  %
        \expandafter\@tempa\expandafter{\thepart  
          \hspace{1em}#1}\else                    
        \addcontentsline{toc}{part}{#1}\fi        
   {\p@rtstyl@                       
    \ifnum \c@secnumdepth >-1\relax        
      {\partintr@styl@\partname\ \thepart  
       \partintr@dot}\partintr@sep\nobreak 
    \fi                                    
    #2\partd@t\markboth{}{}\par}
    \nobreak                       
    \vskip\p@rtafterskip           
   \@afterheading                  
    }                              

\def\@spart#1{{\p@rtcentering\p@rtstyl@                      
    #1\partd@t\par}                 
    \nobreak                        
    \vskip\p@rtafterskip            
    \@afterheading                  
  }                                 

\newif\ifsection@ftind
\newif\ifsection@ftpar

\def\sectionbeforeskip#1{\def\s@ctbeforeskip{#1}}
\def\sectionstyle#1{\def\s@ctstyl@{#1}}
\def\sectiondot#1{\def\sectiond@t{#1}}
\def\sectionafterskip#1{\def\s@ctafterskip{#1}}
\def\sectionintrostyle#1{\def\sectionintr@styl@{#1}}
\def\sectionintro#1{\def\sectionintr@{#1}}
\def\sectionintrodot#1{\def\sectionintr@dot{#1}}
\def\sectionintrosep#1{\def\sectionintr@sep{#1}}
\def\sectionindenttrue{\def\s@ctind{\parindent}}
\def\sectionindentfalse{\def\s@ctind{\z@}}
\def\sectionafterindenttrue{\section@ftindtrue}
\def\sectionafterindentfalse{\section@ftindfalse}
\def\sectionafternewlinetrue{\section@ftpartrue}
\def\sectionafternewlinefalse{\section@ftparfalse}

\newif\ifsubsection@ftind
\newif\ifsubsection@ftpar

\def\subsectionbeforeskip#1{\def\ss@ctbeforeskip{#1}}
\def\subsectionstyle#1{\def\ss@ctstyl@{#1}}
\def\subsectiondot#1{\def\subsectiond@t{#1}}
\def\subsectionafterskip#1{\def\ss@ctafterskip{#1}}
\def\subsectionintrostyle#1{\def\subsectionintr@styl@{#1}}
\def\subsectionintro#1{\def\subsectionintr@{#1}}
\def\subsectionintrodot#1{\def\subsectionintr@dot{#1}}
\def\subsectionintrosep#1{\def\subsectionintr@sep{#1}}
\def\subsectionindenttrue{\def\ss@ctind{\parindent}}
\def\subsectionindentfalse{\def\ss@ctind{\z@}}
\def\subsectionafterindenttrue{\subsection@ftindtrue}
\def\subsectionafterindentfalse{\subsection@ftindfalse}
\def\subsectionafternewlinetrue{\subsection@ftpartrue}
\def\subsectionafternewlinefalse{\subsection@ftparfalse}

\newif\ifsubsubsection@ftind
\newif\ifsubsubsection@ftpar

\def\subsubsectionbeforeskip#1{\def\sss@ctbeforeskip{#1}}
\def\subsubsectionstyle#1{\def\sss@ctstyl@{#1}}
\def\subsubsectiondot#1{\def\subsubsectiond@t{#1}}
\def\subsubsectionafterskip#1{\def\sss@ctafterskip{#1}}
\def\subsubsectionintrostyle#1{\def\subsubsectionintr@styl@{#1}}
\def\subsubsectionintro#1{\def\subsubsectionintr@{#1}}
\def\subsubsectionintrodot#1{\def\subsubsectionintr@dot{#1}}
\def\subsubsectionintrosep#1{\def\subsubsectionintr@sep{#1}}
\def\subsubsectionindenttrue{\def\sss@ctind{\parindent}}
\def\subsubsectionindentfalse{\def\sss@ctind{\z@}}
\def\subsubsectionafterindenttrue{\subsubsection@ftindtrue}
\def\subsubsectionafterindentfalse{\subsubsection@ftindfalse}
\def\subsubsectionafternewlinetrue{\subsubsection@ftpartrue}
\def\subsubsectionafternewlinefalse{\subsubsection@ftparfalse}

\newif\ifparagraph@ftind
\newif\ifparagraph@ftpar

\def\paragraphbeforeskip#1{\def\p@rbeforeskip{#1}}
\def\paragraphstyle#1{\def\p@rstyl@{#1}}
\def\paragraphdot#1{\def\paragraphd@t{#1}}
\def\paragraphafterskip#1{\def\p@rafterskip{#1}}
\def\paragraphintrostyle#1{\def\paragraphintr@styl@{#1}}
\def\paragraphintro#1{\def\paragraphintr@{#1}}
\def\paragraphintrodot#1{\def\paragraphintr@dot{#1}}
\def\paragraphintrosep#1{\def\paragraphintr@sep{#1}}
\def\paragraphindenttrue{\def\p@rind{\parindent}}
\def\paragraphindentfalse{\def\p@rind{\z@}}
\def\paragraphafterindenttrue{\paragraph@ftindtrue}
\def\paragraphafterindentfalse{\paragraph@ftindfalse}
\def\paragraphafternewlinetrue{\paragraph@ftpartrue}
\def\paragraphafternewlinefalse{\paragraph@ftparfalse}

\newif\ifsubparagraph@ftind
\newif\ifsubparagraph@ftpar

\def\subparagraphbeforeskip#1{\def\sp@rbeforeskip{#1}}
\def\subparagraphstyle#1{\def\sp@rstyl@{#1}}
\def\subparagraphdot#1{\def\subparagraphd@t{#1}}
\def\subparagraphafterskip#1{\def\sp@rafterskip{#1}}
\def\subparagraphintrostyle#1{\def\subparagraphintr@styl@{#1}}
\def\subparagraphintro#1{\def\subparagraphintr@{#1}}
\def\subparagraphintrodot#1{\def\subparagraphintr@dot{#1}}
\def\subparagraphintrosep#1{\def\subparagraphintr@sep{#1}}
\def\subparagraphindenttrue{\def\sp@rind{\parindent}}
\def\subparagraphindentfalse{\def\sp@rind{\z@}}
\def\subparagraphafterindenttrue{\subparagraph@ftindtrue}
\def\subparagraphafterindentfalse{\subparagraph@ftindfalse}
\def\subparagraphafternewlinetrue{\subparagraph@ftpartrue}
\def\subparagraphafternewlinefalse{\subparagraph@ftparfalse}

\sectionbeforeskip{\bigskipamount}
\sectionstyle{\large\bf}
\sectiondot{}
\sectionafterskip{.5\bigskipamount}
\sectionintrostyle{}
\sectionintro{}
\sectionintrodot{.}
\sectionintrosep{1.25ex}
\sectionindentfalse
\sectionafterindenttrue
\sectionafternewlinetrue

\subsectionbeforeskip{.8\bigskipamount}
\subsectionstyle{\normalsize\bf}
\subsectiondot{}
\subsectionafterskip{.4\bigskipamount}
\subsectionintrostyle{}
\subsectionintro{}
\subsectionintrodot{.}
\subsectionintrosep{1.25ex}
\subsectionindentfalse
\subsectionafterindenttrue
\subsectionafternewlinetrue

\subsubsectionbeforeskip{.6\bigskipamount}
\subsubsectionstyle{\normalsize\bf}
\subsubsectiondot{}
\subsubsectionafterskip{.3\bigskipamount}
\subsubsectionintrostyle{}
\subsubsectionintro{}
\subsubsectionintrodot{.}
\subsubsectionintrosep{1.25ex}
\subsubsectionindentfalse
\subsubsectionafterindenttrue
\subsubsectionafternewlinetrue

\paragraphbeforeskip{.5\bigskipamount}
\paragraphstyle{\normalsize\bf}
\paragraphdot{.}
\paragraphafterskip{1.25ex}
\paragraphintrostyle{}
\paragraphintro{}
\paragraphintrodot{.}
\paragraphintrosep{1.25ex}
\paragraphindentfalse
\paragraphafterindenttrue
\paragraphafternewlinefalse

\subparagraphbeforeskip{.5\bigskipamount}
\subparagraphstyle{\normalsize\bf}
\subparagraphdot{.}
\subparagraphafterskip{1.25ex}
\subparagraphintrostyle{}
\subparagraphintro{}
\subparagraphintrodot{.}
\subparagraphintrosep{1.25ex}
\subparagraphindenttrue
\subparagraphafterindenttrue
\subparagraphafternewlinefalse

\let\@partoken\par
\long\def\@@gobble#1{}
\def\ignorepar{\@ifnextchar\@partoken{\expandafter\ignorepar\@@gobble}{\ignorespaces}}

\def\@startsection#1#2#3#4#5#6{
   \@tempskipa #4\relax
   \csname if#1@ftind\endcsname\@afterindenttrue\else\@afterindentfalse\fi
   \advance\@tempskipa by\presection
   \if@nobreak \everypar{}\else
     \addpenalty{\@secpenalty}\addvspace{\@tempskipa}%
     \allowbreak\vskip -\presection \fi \@ifstar
     {\@ssect{#1}{#2}{#3}{#4}{#5}{#6}}{\@dblarg{\@sect{#1}{#2}{#3}{#4}{#5}{#6}}}}

\def\@sect#1#2#3#4#5#6[#7]#8{\def\object@type{#1}%
   \ifnum #2>\c@secnumdepth\def\@svsec{}\def\@tempb{}%
      \else\refstepcounter{#1}\def\@svsec{{\csname #1intr@styl@\endcsname%
        {\csname #1intr@\endcsname}\csname the#1\endcsname%
        \csname #1intr@dot\endcsname\kern\csname #1intr@sep\endcsname}}%
        \edef\@tempb{\noexpand\numberline{\csname the#1\endcsname}}\fi%
   \def\@tempa{\addcontentsline{toc}{#1}}%
   \csname if#1@ftpar\endcsname%
      \begingroup #6\relax%
        \@hangfrom{\hskip #3\relax\@svsec}{\interlinepenalty \@M{#8}%
        \csname #1d@t\endcsname\par}%
      \endgroup%
      \csname #1mark\endcsname{#7}%
      \expandafter\@tempa\expandafter{\@tempb #7}%
      \ifautolabel\label*{#8}\fi%
   \else%
      \def\@svsechd{#6\hskip #3\relax%
         \@svsec{#8}%
         \csname #1d@t\endcsname%
         \csname #1mark\endcsname{#7}%
         \expandafter\@tempa\expandafter{\@tempb #7}%
         \ifautolabel\label*{#8}\fi}\fi%
   \@xsect{#1}{#5}\ignorepar}

\def\@ssect#1#2#3#4#5#6#7{%
   \ifnum #2>\c@secnumdepth\def\@tempb{}\else \def\@tempb{\numberline{}}\fi%
     \def\@tempa{\addcontentsline{toc}{s#1}}%
     \csname if#1@ftpar\endcsname
        \begingroup #6\relax
           \@hangfrom{\hskip #3}{\interlinepenalty \@M{#7}%
           \csname #1d@t\endcsname\par}%
        \endgroup
        \csname s#1mark\endcsname{#7}%
        \ifstarredcontents\expandafter\@tempa\expandafter{\@tempb #7}\fi%
        \ifautolabel\label*{#7}\fi%
     \else%
        \def\@svsechd{#6\hskip #3\relax{#7}%
        \csname #1d@t\endcsname%
        \csname s#1mark\endcsname{#7}%
        \ifautolabel\label*{#7}\fi}\fi
   \@xsect{#1}{#5}\ignorepar}

\def\@xsect#1#2{
   \csname if#1@ftpar\endcsname 
       \par \nobreak \vskip #2\relax \@afterheading
    \else \global\@nobreakfalse \global\@noskipsectrue
       \everypar{\if@noskipsec \global\@noskipsecfalse
                   \clubpenalty\@M \hskip -\parindent
                   \begingroup \@svsechd \endgroup \unskip
                   \hskip #2\relax  
                  \else \clubpenalty \@clubpenalty
                    \everypar{}\fi}\fi\ignorespaces}

\def\section{\@startsection{section}{1}{\s@ctind}
  {\s@ctbeforeskip}{\s@ctafterskip}{\s@ctstyl@}}
\def\subsection{\@startsection{subsection}{2}{\ss@ctind}
  {\ss@ctbeforeskip}{\ss@ctafterskip}{\ss@ctstyl@}}
\def\subsubsection{\@startsection{subsubsection}{3}{\sss@ctind}
  {\sss@ctbeforeskip}{\sss@ctafterskip}{\sss@ctstyl@}}
\def\paragraph{\@startsection{paragraph}{4}{\p@rind}
  {\p@rbeforeskip}{\p@rafterskip}{\p@rstyl@}}
\def\subparagraph{\@startsection{subparagraph}{4}{\sp@rind}
  {\sp@rbeforeskip}{\sp@rafterskip}{\sp@rstyl@}}

\def\statementabove#1{\def\th@bove{#1}}
\def\statementstyle#1{\def\thstyl@{#1}}
\def\statementbelow#1{\def\thb@low{#1}}
\def\statementindentfalse{\let\thind@nt\relax}
\def\statementindenttrue{\let\thind@nt\indent}

\def\statementintrostyle#1{\def\thintr@style{#1}}
\def\statementintrodot#1{\def\thintr@dot{#1}}
\def\statementintrosep#1{\def\thintr@sep{#1}}
\def\statementintrobrackets#1#2{\def\thintr@left{#1}\def\thintr@right{#2}}

\statementabove{\medskip}
\statementstyle{\sl}
\statementbelow{\medskip}
\statementindenttrue

\statementintrostyle{\normalshape\bf}
\statementintrodot{.}
\statementintrosep{\kern1.25ex}
\statementintrobrackets{(}{)}

\def\@thskip{\dimen100\lastskip\vskip-\dimen100%
  \th@bove\dimen101\lastskip\vskip-\dimen101%
  \ifdim\dimen100>\dimen101\else\dimen100\dimen101\fi\vskip\dimen100\vskip0pt}

\long\def\@@newtheorem#1#2#3{%
  \newenvironment{#3}%
    {\def\object@type{#3}\par\@thskip%
     \@ifnextchar[{\@enva{#3}{\thstyl@#1{#2}}}{\@envb{#3}{\thstyl@#1{#2}}}}%
    {\end{#3@}}%
  \@ifnextchar[{\@nothm{#3}}{\@nnthm{#3}}}

\def\@nothm#1[#2]#3{%
  \@ifundefined{c@#2}{\@latexerr{No theorem environment `#2' defined}\@eha}%
  {\expandafter\@ifdefinable\csname #1@\endcsname
  {\global\@namedef{the#1}{\@nameuse{the#2}}%
   \global\@namedef{c@#1}{\@nameuse{c@#2}}
   \global\@namedef{p@#1}{\@nameuse{p@#2}}
   \global\@namedef{#1@}{\@nnnthm{#2}{#3}}%
   \global\@namedef{end#1@}{\@endtheorem}}}}

\def\@nnnthm#1#2{\refstepcounter
    {#1}\@ifnextchar[{\@ynnnthm{#1}{#2}}{\@xnnnthm{#1}{#2}}}
 
\def\@xnnnthm#1#2{\@begintheorem{#2}{\csname the#1\endcsname}\ignorespaces}
\def\@ynnnthm#1#2[#3]{\@opargbegintheorem{#2}{\csname the#1\endcsname}{#3}\ignorespaces}

\def\renewtheorem{\@ifnextchar[{\@renewtheorem}{\@renewtheorem[{}{}]}}

\long\def\@renewtheorem[#1]{\@@renewtheorem#1}

\long\def\@@renewtheorem#1#2#3{%
  \expandafter\let\csname#3@\endcsname\undefined
  \renewenvironment{#3}%
    {\def\object@type{#3}\par\@thskip%
     \@ifnextchar[{\@enva{#3}{\thstyl@#1{#2}}}{\@envb{#3}{\thstyl@#1{#2}}}}%
    {\end{#3@}}%
  \@ifnextchar[{\@nothm{#3}}{\@nnthm{#3}}}

\def\@begintheorem#1#2{\@opargbegintheorem{#1}{#2}{}}

\def\@opargbegintheorem#1#2#3{%
        \edef\@tempx{#1}%
        \expandafter\let\expandafter\@tempy#2
        \def\@tempz{#3}%
        \mytrivlist\item[\thind@nt\hskip\labelsep%
        {\thintr@style%
          #1\ifx\@tempx\@empty\else\ifx\@tempy\relax\else\kern1ex\fi\fi#2%
          \ifx\@tempz\@empty%
            \ifx\@tempx\@empty\ifx\@tempy\relax%
            \else\thintr@dot\thintr@sep\fi\else\thintr@dot\thintr@sep\fi%
            \else%
            \ifx\@tempx\@empty\ifx\@tempy\relax\else\kern1ex\fi\else\kern1ex\fi%
           \thintr@left{#3}\thintr@right\thintr@dot\thintr@sep\fi}%
            \hskip-\labelsep]%
        \ifautolabel\label*{#3}\fi}

\def\@endtheorem{\endtrivlist\thb@low}

\def\proofname{Proof}

\def\proofabove#1{\def\pf@bove{#1}}
\def\proofstyle#1{\def\pfstyl@{#1}}
\def\proofbelow#1{\def\pfb@low{#1}}
\def\proofindentfalse{\let\pfind@nt\relax}
\def\proofindenttrue{\let\pfind@nt\indent}

\def\proofintrostyle#1{\def\pfintr@style{#1}}
\def\proofintrodot#1{\def\pfintr@dot{#1}}
\def\proofintrosep#1{\def\pfintr@sep{#1}}
\def\proofintrobrackets#1#2{\def\pfintr@left{#1}\def\pfintr@right{#2}}

\proofabove{\medskip}
\proofstyle{}
\proofbelow{\medskip}
\proofindenttrue

\proofintrostyle{\sl}
\proofintrodot{.}
\proofintrosep{\kern1.25ex}
\proofintrobrackets{of\kern1ex}{}

\def\@pfskip{\dimen100\lastskip\vskip-\dimen100%
  \pf@bove\dimen101\lastskip\vskip-\dimen101%
  \ifdim\dimen100>\dimen101\else\dimen100\dimen101\fi\vskip\dimen100\vskip0pt}

\renewenvironment{proof}%
  {\@pfskip\mytrivlist\item[\pfind@nt]\@ifnextchar[{\pro@f}{\pro@f[\prooftag]}}
  {\ifvoid\provedbox\else\hproved\fi\endtrivlist\pfb@low}

\def\pro@f[#1]{\setbox\provedbox\hbox{\provedboxcontents{#1}}\proofintro{#1}}

\def\proofintro#1{\expandafter\def\expandafter\@tempa\expandafter{#1}%
  {\pfintr@style{\proofname\ifx\@tempa\empty\else\kern1ex\pfintr@left{#1}%
  \pfintr@right\fi}\pfintr@dot\pfintr@sep}\pfstyl@\ignorespaces}

\def\provedmark#1{\def\prm@rk{#1}}
\def\provedsep#1{\def\prs@p{#1}}

\provedmark{$\square$}
\provedsep{\kern1.25ex}

\def\provedtexttrue{\def\prb@x##1{\fbox{\small##1}}}
\def\provedtextfalse{\def\prb@x##1{\prm@rk}}
\def\provedmarkrighttrue{\let\prhf@l\hfill}
\def\provedmarkrightfalse{\let\prhf@l\relax}

\provedtextfalse
\provedmarkrightfalse

\def\provedboxcontents#1{\expandafter\def\expandafter\@tempa\expandafter{#1}%
  \ifx\@tempa\empty\prm@rk\else\prb@x{#1}\fi}

\def\proved{\ifmmode\eqno{\box\provedbox}\else\hproved\fi}

\def\hproved{\unskip\nobreak\prhf@l\penalty50\prs@p\hbox{}\nobreak\prhf@l
  \box\provedbox{\finalhyphendemerits=0\par}}

\def\captionstyle#1{\def\c@ptstyl@{#1}}
\def\captionintrostyle#1{\def\c@pintr@style{#1}}
\def\captionintrodot#1{\def\c@pintr@dot{#1}}
\def\captionintrosep#1{\def\c@pintr@sep{#1}}

\captionstyle{\small\sf}
\captionintrostyle{\bf}
\captionintrodot{.}
\captionintrosep{\hskip1.25ex}

\long\def\@makecaption#1#2{%
    \vskip\captionskip
    \setbox\@tempboxa\hbox{%
      \ifproofing\@ifundefined{the@label}{}
        {\hbox to 0pt{\vbox to 0pt{\vss\hbox{\tiny\the@label}\bigskip}\hss}}\fi
      \c@ptstyl@{\c@pintr@style #1\c@pintr@dot}\ignorespaces #2}%
    \@captionwidth=\hsize \advance\@captionwidth-2\@captionmargin
    \ifdim \wd\@tempboxa >\@captionwidth {%
        \rightskip=\@captionmargin\leftskip=\@captionmargin
        \unhbox\@tempboxa\par}%
      \else
        \hbox to\hsize{\hfil\box\@tempboxa\hfil}%
    \fi}

\def\end@Float#1{%
  \expandafter\caption\expandafter[\the@title]{%
   {\c@pintr@style%
   \ifx\the@caption\empty\ifx\the@title\empty
   \else\c@pintr@sep\fi\else\c@pintr@sep\fi
    \the@title\ifx\the@caption\empty%
     \expandafter\label\expandafter*\expandafter{\the@label}%
    \else\ifx\the@title\empty%
     \expandafter\label\expandafter*\expandafter{\the@label}%
    \else\c@pintr@dot\c@pintr@sep%
     \expandafter\label\expandafter*\expandafter{\the@label}\fi\fi}%
   \ignorespaces\the@caption}%
  \end{#1}}

\renewenvironment{Figure}{\@ifnextchar[%
  {\@myFloat{figure}}{\@myFloat{figure}[htbp]}}{\end@Float{figure}}

\def\@myFloat#1[#2]#3{%
  \def\color@hbox{}\def\color@vbox{}\def\color@endbox{}%
  \begin{#1}[#2]\def\the@label{#3}}

\def\fig#1{\@ifnextchar[{\@fig{#1}}{\@fig{#1}[0pt]}}

\def\@fig#1[#2]#3{\@ifnextchar[{\@@fig{#1}[#2]{#3}}{\@@fig{#1}[#2]{#3}[0pt]}}
\def\@@fig#1[#2]#3[#4]#5#6{%
  \def\the@title{#5}\def\the@caption{#6}\centerline{\fig@{#1}{#2}{#3}}\vskip#4}

\def\fig@@#1#2#3{\leavevmode{\figstyl@\vrule width 0pt height 1.8ex%
 \smash{\framebox{\strut\def\@temp{#1}\ifx\@temp\@empty{ #3 }\else{ #1 }\fi}}}}
\def\fig@@@#1#2#3{\leavevmode\kern#2\epsfbox{#3}}

\def\figstyle#1{\def\figstyl@{#1}}

\figstyle{\fontfamily{cmr}\fontshape{n}\fontseries{m}\selectfont\normalsize}

\newcounter{diagram}

\let\thediagram\theequation
\def\ftype@diagram{2}
\def\ext@diagram{lod}
\def\diagram{\@float{diagram}}
\let\enddiagram\end@float
\newif\if@diagnum

\def\diag#1{\@ifnextchar[{\@diag{#1}}{\@diag{#1}[0pt]}}

\def\@diag#1[#2]#3{\@ifnextchar[{\@@diag{#1}[#2]{#3}}{\@@diag{#1}[#2]{#3}[0pt]}}

\def\@@diag#1[#2]#3[#4]#5{
  \def\the@tag{#5}\@eqnswtrue%
  \centerline{\setbox0\hbox{\diag@{#1}{#2}{#3}}
  \dimen0 -0.5\wd0\dimen1 0.5\ht0\box0%
  \advance\dimen0 0.5\hsize\advance\dimen0 -\rightskip\advance\dimen1 #4%
  \let\@currentlabel\the@tag%
  \setbox0\hbox to 0pt{\hss%
    \fontfamily{cmr}\fontshape{n}\fontseries{m}\selectfont(\the@tag)}%
  \ifx\the@tag\@empty\refstepcounter{equation}\let\@currentlabel\theequation%
    \setbox0\hbox to 0pt{\hss%
      \fontfamily{cmr}\fontshape{n}\fontseries{m}\selectfont(\thediagram)}\fi%
  \if@eqnsw\else\let\@currentlabel\relax\setbox0\hbox to 0pt{}\fi%
  \advance\dimen1 -0.5\ht0%
  \put[\dimen0,\dimen1][l]{%
    \box0\expandafter\label\expandafter*\expandafter{\the@label}\kern0.15em}}}

\def\diag@@#1#2#3{\leavevmode{\diagstyl@\vrule width 0pt height 1.8ex%
 \smash{\framebox{\strut\def\@temp{#1}\ifx\@temp\@empty{ #3 }\else{ #1 }\fi}}}}
\def\diag@@@#1#2#3{\leavevmode\kern#2\epsfbox{#3}}

\def\diagstyle#1{\def\diagstyl@{#1}}

\diagstyle{\fontfamily{cmr}\fontshape{n}\fontseries{m}\selectfont\normalsize}

\def\showfiguresfalse{\let\fig@\fig@@}
\def\showfigurestrue{\let\fig@\fig@@@}

\def\showdiagramsfalse{\let\diag@\diag@@}
\def\showdiagramstrue{\let\diag@\diag@@@}

\showfigurestrue
\showdiagramstrue

\def\n@number{\@eqnswfalse\let\@currentlabel\relax\let\the@tag\relax}

\def\equation{$$
  \@eqnswtrue\def\object@type{equation}\let\nonumber\n@number%
  \advance\c@equation1\edef\@currentlabel{\theequation}\advance\c@equation-1%
  \def\the@tag{\refstepcounter{equation}\eqno\hbox{\@eqnnum}}}

\def\tag#1{\edef\@currentlabel{#1}\def\the@tag{\eqno\hbox{\reset@font\rm(#1)}}}

\def\endequation{\the@tag$$
  \global\@ignoretrue}

\let\it@m\item
\def\item{\@ifnextchar[{\item@}{\item@@}}
\def\item@[#1]{\it@m[#1]\vskip-\lastskip\vskip\itemsep}
\def\item@@{\it@m\vskip-\lastskip\vskip\itemsep}
\def\s@titemsep{\@ifnextchar[{\s@@titemsep}{\relax}}
\def\s@@titemsep[#1]{\itemsep#1}

\let\@itemize\itemize
\let\@enditemize\enditemize
\renewenvironment{itemize}
{\@itemize\itemsep3pt\parsep0pt\topsep0pt\partopsep0pt\s@titemsep}
{\@enditemize\vskip-\lastskip\vskip\itemsep}

\let\@enumerate\enumerate
\let\@endenumerate\endenumerate
\renewenvironment{enumerate}
{\@enumerate\itemsep3pt\parsep0pt\topsep0pt\partopsep0pt\s@titemsep}
{\@endenumerate\vskip-\lastskip\vskip\itemsep}

\let\@description\description
\let\@enddescription\enddescription
\renewenvironment{description}
{\@description\itemsep3pt\parsep0pt\topsep0pt\partopsep0pt\s@titemsep}
{\@enddescription\vskip-\lastskip\vskip\itemsep}

\topsep\dimen200
\partopsep\dimen201

\@definecounter{bibenumi}

\def\thebibliography#1{%
 \section*{\refname}\vskip-\lastskip%
 \list{[\arabic{bibenumi}]}{\topsep0pt\settowidth\labelwidth{[#1]}%
 \leftmargin\labelwidth\advance\leftmargin\labelsep\usecounter{bibenumi}}%
 \def\newblock{\hskip .11em plus .33em minus .07em}%
 \sloppy\clubpenalty4000\widowpenalty4000\sfcode`\.=1000\relax}

\let\@ref@\ref
\let\@pageref@\pageref
\let\@fullref@\fullref
\let\@Fullref@\Fullref
\let\@reftype@\reftype
\let\@Reftype@\Reftype
\let\@label@\label
\let\@cite@\cite
\let\@bibitem@\bibitem

\def\label{\@ifnextchar*{\label@}{\label@{}}}
\def\label@#1#2{\@label@#1{#2}\putl@bel{#2}\ignorespaces}
\def\putl@b@l#1{\put[0pt,.25\baselineskip]{%
  \hbox{\labc@lor{\fontfamily{cmr}\fontshape{n}\fontseries{m}\selectfont%
  \tiny\setbox5\hbox{\vphantom{X}\smash{\ns#1}}%
  \hbox to 0pt{\hss\tiny$\blacktriangledown$\kern-.085em}%
  \raise2.25ex\hbox to 0pt{\hss\framebox{\box5}}}}}}

\def\putr@fl@bel#1{{\let\labc@lor\refc@lor\putl@bel{#1}}}
\def\ref@#1{\@ref@{#1}\putr@fl@bel{#1}}
\def\pageref@#1{\@pageref@{#1}\putr@fl@bel{#1}}
\def\fullref@#1{\@fullref@{#1}\putr@fl@bel{#1}}
\def\Fullref@#1{\@Fullref@{#1}\putr@fl@bel{#1}}
\def\reftype@#1{\@reftype@{#1}\putr@fl@bel{#1}}
\def\Reftype@#1{\@Reftype@{#1}\putr@fl@bel{#1}}

\def\ref@@#1{\leavevmode\refc@lor{\rm$\langle$#1$\rangle$}}
\let\pageref@@\ref@@
\let\Fullref@@\ref@@
\let\fullref@@\ref@@
\let\reftype@@\ref@@
\let\Reftype@@\ref@@

\def\bibitem{\@ifnextchar[{\bibitem@@}{\bibitem@@@}}
\def\bibitem@@[#1]#2{\@bibitem@[#1]{#2}\putl@bel{#2}}
\def\bibitem@@@#1{\@bibitem@{#1}\putl@bel{#1}\ignorespaces}

\def\cit@{\@ifnextchar[{\@cit@@@}{\@cit@@}}
\def\@cit@@#1{\@cite@{#1}{\let\labc@lor\citc@lor\putl@bel{#1}}}
\def\@cit@@@[#1]#2{\@cite@[#1]{#2}{\let\labc@lor\citc@lor\putl@bel{#2}}}
\def\cit@@{\@ifnextchar[{\cit@@@@}{\cit@@@}}
\def\cit@@@#1{\leavevmode{\citc@lor\rm[#1]}}
\def\cit@@@@[#1]#2{\leavevmode{\citc@lor\rm[#2, #1]}}

\def\showcitationstrue{\let\cite\cit@}
\def\showcitationsfalse{\let\cite\cit@@}

\def\showreferencestrue{%
  \let\ref\ref@\let\pageref\pageref@%
  \let\fullref\fullref@\let\Fullref\Fullref@%
  \let\reftype\reftype@\let\Reftype\Reftype@}
\def\showreferencesfalse{%
  \let\ref\ref@@\let\pageref\pageref@@%
  \let\fullref\fullref@@\let\Fullref\Fullref@@%
  \let\reftype\reftype@@\let\Reftype\Reftype@@}

\def\showlabelstrue{\let\putl@bel\putl@b@l}
\def\showlabelsfalse{\let\putl@bel\hid@@}

\def\postit@{\@ifnextchar[{\postit@@}{\p@tp@stit}}
\def\postit@@[#1]{\postit@@@#1,@}
\def\postit@@@#1,{\@ifnextchar{@}{\p@@tp@stit{#1}}{\postit@@@@#1,}}
\def\postit@@@@#1,#2,@{\p@@@tp@stit{#1}{#2}}

\long\def\p@tp@stit#1{\put[0pt,.25\baselineskip]{%
  \hbox{\postitc@lor{\fontfamily{cmr}\fontshape{n}\fontseries{m}\selectfont%
  \tiny\setbox5\hbox{\vphantom{X}\smash{\ns#1}}%
  \hbox to 0pt{\hss\tiny$\blacktriangledown$\kern-.085em}%
  \raise2.25ex\hbox to 0pt{\hss\framebox{\box5}}}}}}

\long\def\p@@tp@stit#1@#2{\put[0pt,.25\baselineskip]{%
  \hbox{\postitc@lor{\fontfamily{cmr}\fontshape{n}\fontseries{m}\selectfont%
  \tiny\setbox5\hbox{\vbox{\hsize#1\leftskip\z@\raggedright
  \parindent\z@{\ns#2\par}\vss}}%
  \hbox to 0pt{\hss\tiny$\blacktriangledown$\kern-.085em}%
  \raise2.25ex\hbox to 0pt{\hss\framebox{\box5}}}}}}

\long\def\p@@@tp@stit#1#2#3{\put[0pt,.25\baselineskip]{%
  \hbox{\postitc@lor{\fontfamily{cmr}\fontshape{n}\fontseries{m}\selectfont%
  \tiny\setbox5\hbox{\vbox to #2{\hsize#1\leftskip\z@\raggedright
  \parindent\z@{\ns#3\par}\vss}}%
  \hbox to 0pt{\hss\tiny$\blacktriangledown$\kern-.085em}%
  \raise2.25ex\hbox to 0pt{\hss\framebox{\box5}}}}}}

\def\postitc@lor{\color{postitcolor}}

\def\showpostittrue{\let\postit\postit@}
\def\showpostitfalse{\let\postit\hid@@@}

\long\def\hid@@#1{\ignorespaces}
\def\hid@@@{\@ifnextchar[{\hid@@@@}{\hid@@}}
\long\def\hid@@@@[#1]{\hid@@}

\makeatother

\papersize{216truemm}{279truemm}
\margins{33truemm}{20.25truemm}
\advance\voffset-3truemm
\advance\hoffset3truemm
\advance\footskip0truemm

\headline{\hfill}
\footline{\small\hfill--\kern1ex\thepage\kern1ex--\hfill}

\sectionbeforeskip{1.5\bigskipamount}
\sectionstyle{\centering\normalsize\sc}
\sectionafterskip{\bigskipamount}
\sectionafterindenttrue

\subsectionbeforeskip{\bigskipamount}
\subsectionstyle{\normalsize\bf}
\subsectionafterskip{.5\bigskipamount}
\subsectionafterindenttrue

\paragraphbeforeskip{\medskipamount}
\paragraphstyle{\ft{cmbxsl10}\boldmath}
\paragraphafterskip{1.25ex}
\paragraphindenttrue
\paragraphdot{.}

\abovedisplayskip\smallskipamount
\belowdisplayskip\smallskipamount
\abovedisplayshortskip\smallskipamount
\belowdisplayshortskip\smallskipamount


\textfloatsep\floatsep

\proofingfalse
\autolabelfalse

\showlabelsfalse
\showcitationstrue
\showreferencestrue

\showfigurestrue
\showdiagramstrue

\newtheorem{stat}{\statname}  \unnumbered{stat}

\newtheorem{nstat}{\nstatname}[section]
\newenvironment{numberedstatement}[1]{\def\nstatname{#1}\begin{nstat}}{\end{nstat}}

\newtheorem[{\ns}{}]{definition}[nstat]{Definition}
\newtheorem{lemma}[nstat]{Lemma}

\newtheorem{theorem}[nstat]{Theorem}
\newtheorem{corollary}[nstat]{Corollary}

\newtheorem{question}[nstat]{Question}
\newtheorem[{\ns}{}]{exercise}[nstat]{Exercise}
\newtheorem[{\ns}{}]{example}[nstat]{Example}
\newtheorem[{\ns}{}]{remark}[nstat]{Remark}

\let\ns\normalshape

\captionstyle{}
\captionintrostyle{}

\showfigurestrue

\sectionbeforeskip{1.5\bigskipamount}
\sectionstyle{\centering\normalsize\bf}
\sectionafterskip{\bigskipamount}
\sectionafterindenttrue

\subsectionbeforeskip{\bigskipamount}
\subsectionstyle{\normalsize\bf}
\subsectionafterskip{.5\bigskipamount}
\subsectionafterindenttrue

\paragraphbeforeskip{\bigskipamount}
\paragraphstyle{\centering\normalsize\sl}
\paragraphafterskip{.5\bigskipamount}
\paragraphafternewlinetrue
\paragraphafterindenttrue
\paragraphdot{}

\statementindenttrue
\statementintrostyle{\sc}
\renewtheorem[{\ns}{}]{remark}[nstat]{Remark}
\renewtheorem[{\ns}{}]{definition}[nstat]{Definition}

\captionstyle{\small}
\captionintrostyle{\sc}

\newcommand{\id}{\mathop{\mathrm{id}}\nolimits}
\newcommand{\Cl}{\mathop{\mathrm{Cl}}\nolimits} 
\newcommand{\Int}{\mathop{\mathrm{Int}}\nolimits} 
\newcommand{\Bd}{\partial} 

\newcommand{\End}{\mathop{\mathrm{End}}\nolimits}
\newcommand{\cs}{\mathop{\#}}

\renewcommand{\:}{\,{:}\;}
\newcommand{\CP}{{C\mkern-1.5muP}}
\newcommand{\CPbar}{{\vphantom{CP}\smash{\widebar{C\mkern-1.5muP}}}}
\newcommand{\simtimes}{\mathbin{\widetilde{\smash{\times}}}}
\newcommand{\ks}{\mathop{\mathrm{ks}}\nolimits}

\def\(#1\){$(${\sl #1}\/$)$}

\let\emptyset\varemptyset

\def\emph#1{{\sl #1}\/}

\begin{document}

\title{\large\bf ON BRANCHED COVERING\\REPRESENTATION OF 4-MANIFOLDS}
\author{
\sc\normalsize Riccardo Piergallini\\
\sl\normalsize Scuola di Scienze e Tecnologie\\[-3pt]
\sl\normalsize Universit\`a di Camerino -- Italy\\
\tt\small riccardo.piergallini@unicam.it
\and
\sc\normalsize Daniele Zuddas\\
\sl\normalsize Lehrstuhl Mathematik VIII\\[-3pt]
\sl\normalsize Universit\"{a}t Bayreuth -- Germany\\
\tt\small zuddas@uni-bayreuth.de
}
\date{}

\vglue-9pt
\maketitle
\vskip-18pt

\begin{abstract}
\baselineskip13.5pt
\medskip

Assuming $M$ to be a connected oriented PL 4-manifold, our main results are the following: (1) if $M$ is compact  with (possibly empty) boundary, there exists a simple branched covering $p \: M \to S^4 - \Int(B^4_1 \cup \dots \cup B^4_n)$, where the $B^4_i$'s are disjoint PL 4-balls, $n \geq 0$ is the number of boundary components of $M$; (2) if $M$ is open, there exists a simple branched covering $p \: M \to S^4 - \End M$, where $\End M$ is the end space of $M$ tamely embedded in $S^4$.

In both cases, the degree $d(p)$ and the branching set $B_p$ of $p$ can be assumed to satisfy one of these conditions: (1) $d(p) \mkern5mu{=}\mkern5mu 4$ and $B_p$ is a properly self-transversally immersed locally flat PL surface; (2) $d(p) \mkern5mu{=}\mkern5mu 5$ and $B_p$ is a properly embedded locally flat PL surface. 
In the compact (resp. open) case, by relaxing the assumption on the degree we can have $B^4$ (resp. $R^4$) as the base of the covering.

A crucial technical tool used in all the proofs is a quite delicate cobordism lemma for coverings of $S^3$, which also allows us to obtain a relative version of the branched covering representation of bounded 4-manifolds, where the restriction to the boundary is a given branched covering.

We also define the notion of branched covering between topological manifolds, which extends the usual one in the PL category. In this setting, as an interesting consequence of the above results, we prove that any closed oriented \emph{topological} 4-manifold is a 4-fold branched covering of $S^4$. According to almost-smoothability of 4-manifolds, this branched covering could be wild at a single point.

\medskip\smallskip\noindent
{\sl Keywords}\/: branched coverings, wild branched coverings, 4-manifolds.

\medskip\noindent
{\sl AMS Classification}\/: 57M12, 57M30, 57N13.

\end{abstract}

\section*{Introduction}

In \cite{Mo78}, Montesinos proved that any oriented 4-dimensional 2-handlebody is a 3-fold simple covering of $B^4$ branched over a ribbon surface. In \cite{Pi95}, based on this result and on covering moves for 3-manifolds (see \cite{Pi91}), the first author proved that every closed connected oriented PL 4-mani\-fold $M$ is a four-fold simple covering of $S^4$ branched over an immersed locally flat PL surface, possibly having a finite number of transversal double points. 
Subsequently, Iori and Piergallini \cite{IP02} showed that the double points of the branch set can be removed after stabilizing the covering with an extra fifth sheet, in order to get an embedded locally flat PL surface.
This partially solves Problem 4.113 (A) of Kirby's list \cite{Ki95}, but it is still unknown whether double points of the branch set can be removed without stabilization.

It is then natural to ask whether such results can be generalized to arbitrary compact 4-manifolds with (possibly disconnected) boundary and to open 4-manifolds. Moreover, it is intriguing to explore what we can do in the TOP category, namely for compact topological 4-manifolds.

The aim of the present article is to answer these questions. This can be done in light of the results obtained by Bobtcheva and the first author in \cite{BP12} (see also \cite{BP05}),\break about covering moves relating different branched coverings of $B^4$ having PL homeomorphic covering spaces.

In the PL category, we prove Theorems \ref{bc-cpt/thm} and \ref{bc-cpt-bis/thm} below in the compact case, as well as Theorems \ref{bc-open/thm} and \ref{bc-open-bis/thm} in the open case.
Then, by compactifying coverings, we obtain Theorem \ref{bc-top/thm}, which provides a similar representation result for \emph{topological} 4-manifolds in terms of (possibly wild) topological branched coverings, according to Definitions \ref{bc-top/def} and  \ref{bc-wild/def}.

These results were inspired by Guido Pollini's PhD thesis \cite{Po07}, written under the advise of the first author. We are grateful to Guido for his contribution.

A key ingredient in our arguments is the fact that, for any two $d$-fold simple coverings $p_0,p_1 \: M \to S^3$ branched over links, with $d \geq 4$, there exists a $d$-fold simple cobordism covering $p \: M \times [0,1] \to S^3 \times [0,1]$ branched over a self-transversally immersed (embedded for $d \geq 5$) locally flat PL surface, whose restrictions over $S^3 \times \{0\}$ and $S^3 \times \{1\}$ coincide with $p_0 \times \id_{\{0\}}$ and $p_1 \times \id_{\{1\}}$, respectively, provided $p_0$ and $p_1$ are \emph{ribbon fillable}, a technical condition explained in Definition \ref{ribbon-ext/def}.

The existence of such cobordism branched covering follows as a special case of Theorem \ref{bc-cpt/thm} and it is used in the proofs of Theorems \ref{bc-cpt-bis/thm}, \ref{bc-open/thm} and \ref{bc-open-bis/thm}. On the other hand, the proof of Theorem \ref{bc-cpt/thm} depends on the weaker version of the above cobordism property represented by Lemma \ref{concordance/thm}, in which the restriction of $p$ over $S^3 \times \{1\}$ is only PL equivalent but not necessarily equal to $p_1 \times \id_{\{1\}}$.

In \cite{PZ17-2} we use Theorem \ref{bc-cpt/thm} to characterize the PL 4-ma\-nifolds that are branched coverings of one of the following manifolds: $\CP^2$, $\CPbar^2$, $S^2 \times S^2$, $S^2 \simtimes S^2$, or $S^3 \times S^1$. Therein, we derive also representation results for submanifolds as branched coverings of standard submanifolds of such basic 4-manifolds.

We will always adopt the PL point of view if not differently stated, referring to the book of Rourke and Sanderson \cite{RS72} for the basic definitions and facts concerning PL topology. However, all our results in the PL category also have a smooth counterpart, being PL $=$ DIFF in dimension four.

\section{Definitions and results in the PL category\label{stats/sec}}

We recall that a \emph{branched covering} $M \to N$ between compact PL manifolds is defined as a non-degenerate PL map that restricts to a (finite degree) ordinary covering over the complement of a codimension two closed subpolyhedron of $N$.\break This is the usual specialization to compact PL manifolds of the very general topological notion of branched covering introduced by Fox in his celebrated paper \cite{Fo57} (see also Montesinos \cite{Mo05}).

First of all, we extend the above definition to non-compact PL manifolds. In doing so, we also remove the finiteness assumption on the degree. This will be useful in Theorem \ref{bc-open-bis/thm}, where we need infinitely many sheets.

\begin{definition}\label{bc-pl/def}
We call a non-degenerate PL map $p \: M \to N$ between PL\break $m$-manifolds with (possibly empty) boundary a $d$-fold \emph{branched covering}, provided the following two properties are satisfied: (1) every $y \in N$ has a compact connected neighborhood $C \subset N$ such that all the connected components of $p^{-1}(C)$ are compact; (2) the restriction $p_| \: M - p^{-1}(B_p) \to N - B_p$ over the complement of an $(m-2)$-di\-mensional closed subpolyhedron $B_p \subset N$ is an ordinary covering of degree $d \leq \infty\,$.
\end{definition}

More precisely, by $B_p$ we denote the minimal subpolyhedron of $N$ satisfying property (2), which is homogeneously $(m-2)$-dimensional, that is each top cell of it has dimension $m-2$. This is unique and is called the \emph{branch set} of the branched covering $p$.
The degree $d = d(p)$ coincides with the maximum cardinality of the fibers $p^{-1}(y)$ with $y \in N$ and it is called the \emph{degree} of the branched covering $p$. In fact, when\break $d(p)$ is finite, then $y \in B_p$ if and only if $p^{-1}(y)$ has cardinality less than $d(p)$. 

We remark that property (1) in the above definition implies (and, in our situation, it is equivalent to) the completeness of $p$ in the sense of Fox \cite{Fo57}, therefore $p$ is the Fox completion of its restriction $p_| \: M - p^{-1}(B_p) \to N$ (cf. Montesinos \cite{Mo05}). As such, $p$ is completely determined, up to PL homeomorphisms, by the inclusion $B_p \subset N$ and by the ordinary covering $p_| \: M - p^{-1}(B_p) \to N - B_p$, or equivalently, by the associated \emph{monodromy} homomorphism $\omega_p \: \pi_1(N - B_p) \to \Sigma_{d(p)}$.
Finally, $p$ is called a \emph{simple} branched covering if the monodromy $\omega_p(\mu)$ of any meridian $\mu \in \pi_1(N - B_p)$ around $B_p$ is a transposition (in general, it decomposes into disjoint cycles of finite order). We recall that \emph{meridians} \label{meridians} around $B_p$ are only defined at the locally flat points of $B_p$, as the loops obtained by a concatenation of the form $\mu = aca^{-1}$, where $c$ is a loop parametrizing the boundary of a small locally flat PL disk transversal to $B_p$ and $a$ is a path from the base point of $N - B_p$ to the base point of $c$.

In the special case when $N$ is simply connected, the group $\pi_1(N - B_p)$ is generated by a suitable set of meridians, such as a Hurwitz system in dimension 2 or a Wirtinger set of generators in dimension 3 and 4, and the monodromy can be encoded by labeling (a diagram of) $B_p$ with the transpositions corresponding to these meridians.

According to the above definitions and notations, we collect the results mentioned in the introduction in the following statement.

\begin{theorem}[\ns\cite{Pi95,IP02}]\label{bc-clo/thm}
Every closed connected oriented PL $4$-manifold $M$ can be represented by a simple branched covering $p \: M \to S^4$, with degree $d(p)$ and branch set $B_p \subset S^4$ satisfying one of the following conditions:
\begin{enumerate}
\item[\(\rlap{a}{\phantom{b}}\)] \emph{$d(p) \mkern5mu{=}\mkern5mu 4$ and $B_p$ is a self-transversally immersed locally flat PL surface;}
\item[\(b\)] \emph{$d(p) \mkern5mu{=}\mkern5mu 5$ and $B_p$ is an embedded locally flat PL surface.}
\end{enumerate}
\vskip-\lastskip\vskip0pt
\end{theorem}

Next theorems represent the extensions of the previous one to bounded and open 4-manifolds, respectively, that we will prove in this paper. In order to state and prove them, we recall the notion of ribbon surface in $B^4$ and introduce the ribbon fillability property for simple branched coverings of $S^3$.

A properly embedded PL surface $S \subset B^4$ is a \emph{ribbon surface} if and only if it can be realized by pushing inside $B^4$ the interior of a PL immersed surface $S' \subset S^3 = \Bd B^4$, whose only self-intersections consist of transversal double arcs like the one depicted in Figure \ref{doublearc/fig}. Up to PL isotopy of ribbon surfaces in $B^4$ the surface $S$ is uniquely determined by the surface $S'$, which is called the 3-dimensional diagram of $S$, and in the Figures we will always draw the latter to represent the former.

\begin{Figure}[htb]{doublearc/fig}
\fig{}{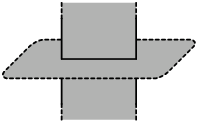}
    {}{A self-intersection arc in the diagram of a ribbon surface.}
\vskip-3pt
\end{Figure}

\begin{definition}\label{ribbon-ext/def}
A simple branched covering $p \: M \to S^3$ is defined to be {\sl ribbon fillable} if it can be extended to a simple branched covering $q \: W \to B^4$ whose branch set $B_q \subset B^4$ is a ribbon surface (which immediately implies that $M = \Bd W$, $B_p = \Bd B_q \subset  S^3$ is a link, and $d(p) = d(q)$). For the sake of convenience, we also call ribbon fillable any simple branched covering $p \: M \to S^3_1 \cup \dots \cup S^3_n$\break that is a disjoint union of ribbon fillable coverings.
\end{definition}

We observe that the above definition is invariant under equivalence of $p$ up to PL homeomorphisms. Hence, ribbon fillability of $p\:M \to S^3$ can be expressed in terms of the labeled branch set $B_p$ by requiring that it is a labeled link in $S^3$ bounding a labeled ribbon surface in $B^4$.

When using simple branched coverings of $S^3$ to represent closed connected oriented 3-manifolds, ribbon fillability arises quite naturally and it is not so restrictive. In fact, it is satisfied by any branched covering representation of such a 3-manifold derived from an integral surgery description of it by the procedure given in Montesinos \cite{Mo78} (cf. also Edmonds \cite{Ed78}) or by the more effective one provided in Bobtcheva and Piergallini \cite{BP05,BP12} (see Section \ref{proofs/sec} below).

\begin{theorem}\label{bc-cpt/thm}
Every compact connected oriented PL $4$-manifold $M$ with $n$ boundary components can be represented by a simple branched covering $p \: M \to S^4 - \Int (B^4_1 \cup \dots \cup B^4_n)$ satisfying property \(a\) or \(b\) as in Theorem \ref{bc-clo/thm}, with the $B^4_i$'s pairwise disjoint standard PL $4$-balls in $S^4$ and $B_p$ a bounded surface properly immersed or embedded in $S^4 - \Int (B^4_1 \cup \dots \cup B^4_n)$. Moreover, the restriction of $p$ to the boundary can be required to coincide with any given ribbon fillable simple branched covering $b \: \Bd M \to \Bd B^4_1 \cup \dots \cup \Bd B^4_n$ with $d(b) = d(p)$.
\end{theorem}

In Theorem \ref{bc-cpt/thm}, if the boundary is connected and non-empty, that is $n=1$, we have a simple branched covering $p \: M \to B^4$. By relaxing the constraint on the degree, we can always require that the base of the covering be $B^4$, even if $M$ has more than one boundary component.

\begin{theorem}\label{bc-cpt-bis/thm}
Every compact connected oriented PL $4$-manifold $M$ with\break $n \geq 2$ boundary components is a $3n$-fold simple covering of $B^4$ branched over a properly embedded locally flat PL surface in $B^4$. Moreover, the restriction of the covering to the boundary can be required to coincide with any given $3n$-fold ribbon fillable simple branched covering of $S^3$.
\end{theorem}

For a non-compact manifold $M$, we denote by $\End M$ the \emph{end space} of $M$, that is the inverse limit of the inclusion system of component spaces $C(M - K)$ with $K$ varying on the compact subspaces $K \subset M$ (see Freudenthal \cite{Fr45}). Since $\End M$ is a compact totally disconnected metrizable space, possibly containing a Cantor set, it can be embedded in $R$.

\begin{theorem}\label{bc-open/thm}
Every open connected oriented PL $4$-manifold $M$ can be represented by a simple branched covering  $p \: M \to S^4 - \End M$ satisfying property \(a\) or \(b\) as in Theorem \ref{bc-clo/thm}, with $\End M$ embedded in $S^4$ as a tame totally disconnected subspace (in particular, we can have $\End M \subset S^1 \subset S^4$) and $B_p$ an unbounded surface properly immersed or embedded in $S^4 - \End M$.
\end{theorem}

In the special case when $M$ has only one end, that is $\End M$ consists of a single point, then Theorem \ref{bc-open/thm} tells us that $M$ is a simple branched covering of $R^4$.\break As a direct consequence we have the following corollary.

\begin{corollary}\label{exoticR4/thm}
For every exotic $R^4_{\mathrm{ex}}$ there is a simple branched covering $p \: R^4_{\mathrm{ex}} \to R^4$ to the standard $R^4$, satisfying property \(a\) or \(b\) as in Theorem~\ref{bc-clo/thm}.
\end{corollary}

In the same spirit of Theorem \ref{bc-cpt-bis/thm}, we have a similar result for open 4-manifolds. Namely, by relaxing the constraint on the degree as above, we can always require that the base of the covering is $R^4$, even if $M$ has more than one end.

\begin{theorem}\label{bc-open-bis/thm}
Every open connected oriented PL $4$-manifold $M$ with more than one end is a $3n$-fold simple covering of $R^4$ branched over a properly embedded locally flat PL surface in $R^4$, with $n = \min\{\aleph_0,\left|\End M\right|\}$.
\end{theorem}

The theorems above can be combined in various ways, by including in a single statement different points of view. In particular, we limit ourselves to consider next Theorems \ref{bc-cpt-open/thm} and \ref{cobordism-bis/thm} below. The former includes Theorems \ref{bc-cpt/thm} and \ref{bc-open/thm}, while the latter includes Theorems \ref{bc-cpt/thm} and \ref{bc-cpt-bis/thm}, as well as Lemma \ref{concordance/thm} stated in Section \ref{proofs/sec}. The proofs of these new theorems are nothing else than combinations of the proofs of the constituent ones, so we leave them to the reader.

\begin{theorem}\label{bc-cpt-open/thm}
For every connected oriented PL 4-manifold $M$ with (possibly empty) compact boundary, there exists a simple branched covering $p \: M \to S^4 - \Int(B^4_1 \cup \dots \cup B^4_n) - \End M$ satisfying property \(a\) or \(b\) as in Theorem \ref{bc-clo/thm}, where the $B^4_i$'s are pairwise disjoint locally flat PL 4-balls in $S^4$, $n \geq 0$ is the number of boundary components of $M$, and $\End M$ is the (possibly empty) end space of $M$ tamely embedded in $S^4 - \Int(B^4_1 \cup \dots \cup B^4_n)$.
\end{theorem}

\begin{theorem}\label{cobordism-bis/thm}
Let $M$ be a compact connected oriented PL 4-manifold with boundary and $b \: \Bd M \to S^3_1 \cup \dots \cup S^3_k$ be a $d$-fold ribbon fillable simple branched covering over a disjoint union of 3-spheres, with $k \geq 1$ and $d \geq 4$ (resp. $d \geq 5$). Then $b$ can be extended to a $d$-fold simple branched covering $p \: M \to S^4 - \Int(B^4_1 \cup \dots \cup B^4_k)$, whose branch set $B_p$ is a properly self-transversally immersed (resp. a properly embedded) locally flat PL surface.
\end{theorem}

\section{Branched coverings in the TOP category}\label{BCtop/sec}

In order to deal with topological 4-manifolds, we need a more general notion of branched covering, not requiring PL structures and admitting a possibly wild branch set.

\pagebreak

\begin{definition}\label{bc-top/def}
We call a continuous map $p \: M \to N$ between topological\break $m$-manifolds with (possibly empty) boundary a \emph{tame topological branched covering} if it is locally modeled on PL branched coverings, meaning that for every $y \in N$ there exists a local chart $V$ of $N$ at $y$ and pairwise disjoint local charts $U_i$ of $M$ at all the $x_i \in p^{-1}(y)$, such that $p^{-1}(V) = U = \cup_i U_i$ and $p_| \: U \to V$ is a PL branched covering.
\end{definition}

The local chart $V$ in the above definition can be replaced by an $m$-ball $C$ centered at $y$ such that $p^{-1}(C) = \cup_iC_i$ is the union of pairwise disjoint $m$-balls, each $C_i$ being centered at a point $x_i$ of $p^{-1}(y)$ and each restriction $p_| \: C_i \to C$ being\break topologically equivalent to the cone of a PL branched covering $S^{m-1} \to S^{m-1}$.\break Using this local conical structure, one could also define the notion of topological branched covering by induction on the dimension $m$, starting with ordinary coverings for $m = 1$.

As an immediate consequence of the existence of the local models, a tame topological branched covering $p$ is a discrete open map. Furthermore, the union of all the branch sets of the local restrictions over charts $V$ as in the definition is an\break $(m-2)$-dimensional (locally tame) subspace $B_p \subset N$, which we call the \emph{branch set} of $p$, and the restriction $p_| \: M - p^{-1}(B_p) \to N - B_p$ over the complement of $B_p$ is an ordinary covering of degree $d(p) \leq \infty$, which we call the \emph{degree} of $p$. So, $p$ satisfies property (2) as in Definition \ref{bc-pl/def}, but with $B_p$ being a polyhedron only locally.

On the other hand, $p$ turns out to be complete, satisfying the condition (1) as in Definition \ref{bc-pl/def}, hence it is the Fox completion of $p_| \: M - p^{-1}(B_p) \to N$ (cf. Fox \cite{Fo57} or Montesinos \cite{Mo05}). Thus, like in the PL case, $p$ is completely determined, up to homeomorphisms, by the inclusion $B_p \subset N$ and by the \emph{monodromy} homomorphism $\omega_p \: \pi_1(N-B_p) \to \Sigma_{d(p)}$. Moreover, it still makes sense to speak of \emph{meridians} around $B_p$ (based on the PL structure of local models, the same notion of meridian recalled at page \pageref{meridians} after Definition \ref{bc-pl/def} still works here), and to call $p$ \emph{simple} if the monodromy of each meridian is a transposition.

\begin{definition}\label{bc-wild/def}
We call a continuous map $q \: M \to N$ between topological $m$-manifolds with (possibly empty) boundary a \emph{wild topological branched covering} if it is discrete and open,
$q^{-1}(\Bd N) = \Bd M$, and the following two conditions hold: (1) every $y \in N$ has a compact connected neighborhood $C \subset N$ such that all the connected components of $q^{-1}(C)$ are compact; (2) the restriction $p = q_| \: M - q^{-1}(W_q) \to N - W_q$ over the complement of a closed nowhere dense subspace $W_q \subset N$ is a tame topological branched covering.
\end{definition}

We always assume $W_q$ to be minimal with the property required in the above definition, and call it the \emph{wild set} of $q$. Of course $q$ is actually wild only if $W_q \neq \emptyset$, otherwise it is a tame topological branched covering.

For a wild topological branched covering $q \: M \to N$, with $p$ its tame restriction as in the definition, we call $B_q = W_q \cup B_p$ the \emph{branch set} of $q$ and $d(q) = d(p)$ the \emph{degree} of $q$. By the minimality of $W_q$ and $B_p$, we have $B_q = q(S_q)$, with $S_q \subset M$ denoting the \emph{singular set} of $q$, that is the set of points of $M$ where $q$ is not a local homeomorphism. Then, Theorem 2 of Church \cite{Ch78} applies to give the following estimate for the Lebesgue covering dimension: $\dim S_q = \dim B_q \leq m-2$. This easily implies that $\dim q^{-1}(B_q) \leq m-2$ as well. Therefore, $N - B_q$ and $M - q^{-1}(B_q)$ are dense and locally connected in $N$ and $M$, respectively, and so we can conclude that $q$ is the Fox completion of the restriction $q_| \: M - q^{-1}(B_q) \to N$.\break Since $q_| \: M - q^{-1}(B_q) \to N - B_q$ is an ordinary covering, $q$ is a branched covering in the sense of Fox \cite{Fo57} (for $M$ connected) and Montesinos \cite{Mo05}, and it is completely determined, up to topological equivalence, by the inclusion $B_q \subset N$ and the \emph{monodromy} $\omega_q = \omega_p \: \pi_1(N - B_q) \to \Sigma_{d(q)}$.

In the special case when $M$ and $N$ are compact and $\dim W_q = 0$, according to Montesinos in \cite[Theorem 2]{Mo02}, the Fox \emph{compactification theorem} \cite[pag. 249]{Fo57} can be generalized to see that  $q$ is actually the Freudenthal \emph{end compactification} (see \cite{Fr45}) of its restriction $p$ over $N - W_q$. In particular, $M$ and $N$ are the end compactifications of $M - q^{-1}(W_q)$ and $N - W_q$, respectively, hence $q^{-1}(W_q) \cong \End(M - q^{-1}(W_q))$ and $W_q \cong \End(N - W_q)$.

In light of the above definitions and recalling that any open 4-manifold admits a PL structure (is smoothable) by a theorem of Lashof \cite{La71} (see also Freedman and Quinn \cite{FQ90}), we can state our third theorem about the branched covering representation of topological 4-manifolds. 

\begin{theorem}\label{bc-top/thm}
Every closed connected oriented topological $4$-manifold $M$ can be represented by a topological branched covering $q \: M \to S^4$, which is the one-point compactification of a simple PL branched covering of $R^4$ satisfying property \(a\) or \(b\) as in Theorem \ref{bc-clo/thm}. Then, the branch set $B_q$ is the one-point compactification of a surface in $R^4$ and the wild set $W_q$ consists of at most a single point.
\end{theorem}

\section{Proofs\label{proofs/sec}}

Our starting point is the branched covering representation of  compact connected oriented 4-dimensional 2-handlebodies up to 2-deformations that is provided in Bobtcheva and Piergallini \cite{BP05,BP12}. As usual, here and in the following, we call a 2-handlebody any handlebody whose handles all have index $\leq 2$, and a 2-deformation any sequence of handle operations (isotopy, sliding and addition/deletion of canceling handles) not involving any handle of index $> 2$.

Below we briefly recall the procedure described in \cite[Section 3]{BP05} (see also \cite[Sections 6.1 and 3.4]{BP12}), for deriving from any Kirby diagram $K$ of a connected oriented\break 4-dimensional 2-handlebody $H$ a labeled ribbon surface $S_K \subset B^4$ representing a simple 3-fold covering $p \: H \to B^4$ branched over $S_K$.

\begin{Figure}[htb]{ribbons/fig}
\fig{}{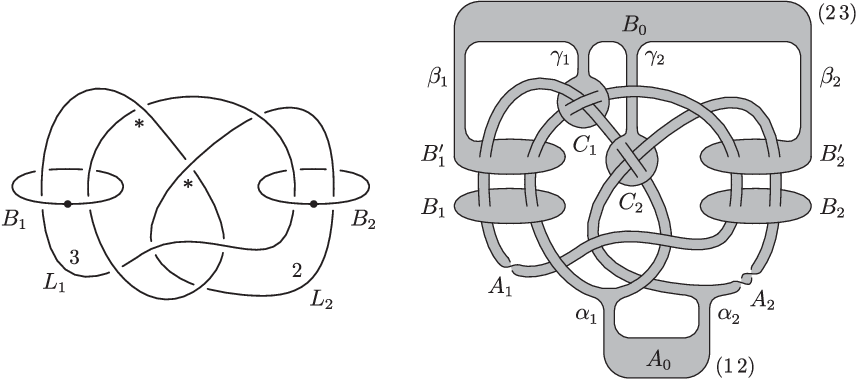}
    {}{A Kirby diagram $K$ and the labeled ribbon surface $S_K$.}
\end{Figure}

Let $K \subset R^3 \cong S^3 - \{\infty\}$ be any Kirby diagram of an oriented 4-dimensional 2-handlebody $H \cong H^0 \cup H^1_1 \cup \dots \cup H^1_m \cup H^2_1 \cup \dots \cup H^2_n$ with a single 0-handle $H^0$, 1-handles $H^1_i$ and 2-handles $H^2_j$. Denote by $B_1,\dots,B_m \subset R^3$ the disjoint disks spanned by the dotted unknots of $K$ representing the 1-handles and by\break $L_1,\dots,L_n \subset R^3$ the framed components of $K$ representing the attaching maps of the 2-handles. Then, the labeled ribbon surface $S_K \subset R^4_+ \cong B^4 -\{\infty\}$ can be constructed as follows (look at Figure \ref{ribbons/fig}, where a simple Kirby diagram $K$ and the corresponding labeled ribbon surface $S_K$ are drawn, respectively on the left and on the right side).

\begin{numberedstatement}{Procedure}[Construction of $S_K$]
\label{SK/proc}\ns

\begin{enumerate}
\item[]{}
\item[1)] Choose a trivializing set of crossings in the diagram of the framed link $L = L_1 \cup \dots \cup L_n$ (the asterisked ones in Figure \ref{ribbons/fig}); denote by $L' = L'_1 \cup \dots \cup L'_n \subset R^3$ the trivial link obtained by inverting those crossings, and by $D_1,\dots,D_n \subset R^3$ a family of disjoint disks spanned by $L'_1,\dots,L'_n$, respectively;
\item[2)] let $A_1,\dots,A_n \subset R^3$ be a family of disjoint (possibly
non-orientable) narrow closed bands, each $A_j$ having $L'_j$ as the core and representing half the framing that $L'_j$ inherits from $L_j$ (by parallel transport at the crossing changes);
\item[3)] let $B'_1,\dots,B'_m \subset R^3$ be a family of disjoint disks, each $B'_i$ being parallel to $B_i$;
\item[4)] let $C_1,\dots,C_\ell \subset R^3$ be a family of disjoint small disks, each $C_k$ being placed at one of the trivializing crossings and forming with the involved bands $A_j$ a fixed pattern of ribbon intersections inside a 3-ball thickening of it, as in Figure \ref{ribbons/fig};
\item[5)] choose a family of disjoint narrow bands $\alpha_1,\dots,\alpha_n \subset R^3$, each $\alpha_j$ connecting $A_j$ to a fixed disk $A_0$ disjoint from all the other disks and bands, with the only constraints that it cannot meet any disk $D_1,\dots,D_n$, the 3-ball spanned by any pair of parallel disks $B_i$ and $B'_i$, and the 3-ball thickening of any $C_k$;
\item[6)] choose a family of disjoint narrow bands $\beta_1,\dots,\beta_m \subset R^3$, each $\beta_i$ connecting $B'_i$ to a fixed disk $B_0$ disjoint from all the other disks and bands, with the same constraints as above;
\item[7)] choose a family of disjoint narrow bands $\gamma_1,\dots,\gamma_\ell \subset R^3$, each $\gamma_k$ connecting $C_k$ to the disk $B_0$, with the same constraints as above;
\item[8)] put $A = A_0 \cup_{j=1}^n (\alpha_j \cup A_j) \subset R^3$ and $B = B_0 \cup_{i=1}^m (\beta_i \cup B'_i) \cup_{k=1}^\ell (\gamma_k \cup C_k) \subset R^3$;
\item[9)] then, $S_K \subset R^4_+ \subset B^4$ is the ribbon surface whose 3-dimensional diagram is given by $A \cup B \cup B_1 \cup \dots \cup B_m$; in other words, $S_K$ is obtained by pushing the interior of the connected surfaces $A,B,B_1,\dots,B_m$ inside the interior of $R^4_+$, in such a way that all the ribbon intersections (formed by $A$ passing through\break $B \cup B_1 \cup \dots \cup B_m$) disappear;
\item[10)] finally, the labeling of $S_K$ giving the monodromy of the simple 3-fold branched covering $p \: H \to B^4$ is the one determined by assigning the transpositions $(1\;2)$ and $(2\;3)$ to the standard meridians of $A_0$ and $B_0 \cup B_1 \cup \dots \cup B_m$, respectively, in the 3-dimensional diagram of $S_K$.
\end{enumerate}
\end{numberedstatement}

The construction above depends on various choices, the significant ones being in steps 1, 5, 6 and 7. However, the labeled ribbon surfaces obtained from different choices become equivalent up to labeled isotopy of ribbon surfaces in $B^4$ (called 1-isotopy in Bobtcheva and Piergallini \cite{BP05,BP12}) and the covering moves $R_1$ and $R_2$ depicted in Figure \ref{rmoves/fig}, after adding to them a separate trivial disk with label $(3\;4)$.
We recall that the addition of such disk represents the stabilization of the branched covering $p$ with an extra trivial fourth sheet to give a simple 4-fold branched covering $\widetilde p \: H \cong  H \cs_{\Bd} B^4 \to B^4$.

The labels $a,b,c$ and $d$ in Figure \ref{rmoves/fig}, as well as in Figures
\ref{kirby/fig} to \ref{cusps/fig}, are assumed to be pairwise distinct.

\begin{Figure}[htb]{rmoves/fig}
\fig{}{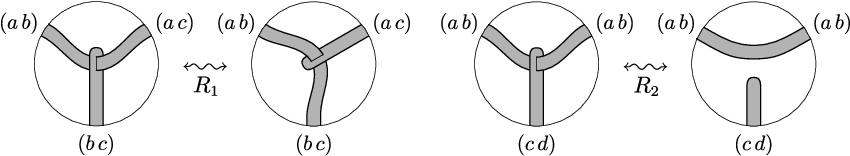}
    {}{Covering moves for labeled ribbon surfaces.}
\end{Figure}

The covering space $H$ of any simple branched covering $p \: H \to B^4$ described by a labeled ribbon surface $S \subset B^4$ is a 4-dimensional 2-handlebody whose handle structure is uniquely determined, up to 2-deformations, by the ribbon structure of $S$. Moreover, the following equivalence theorem holds (Theorem 1 in \cite{BP05}, Theorem 6.1.5 in \cite{BP12}).

\begin{theorem}[\ns\cite{BP05,BP12}]\label{equiv-ribbon/thm}
Let $S$ and $S'$ be two labeled ribbon surfaces in $B^4$ representing compact connected oriented $4$-dimensional $2$-handlebodies as simple branched coverings of $B^4$ of the same degree $\geq 4$. Then, $S$ and $S'$ are related by labeled isotopy of ribbon surfaces and the moves $R_1$ and $R_2$ in Figure \ref{rmoves/fig} if and only if the handlebodies they represent are equivalent up to $2$-deformations.
\end{theorem}

For the purposes of this paper, we need to consider the implication of the above theorem on the boundary. This implication is stated in a precise way in the next theorem, which is a restatement of Theorem 2 in \cite{BP05}, or Theorem 6.1.8 in \cite{BP12}. In fact, handle trading and blow-up moves (see Figure \ref{kirby/fig}), introduced therein in order to interpret the Kirby calculus for 3-manifolds in terms of labeled ribbon surfaces, reduce to isotopy when restricted to the boundary.

\begin{Figure}[htb]{kirby/fig}
\fig{}{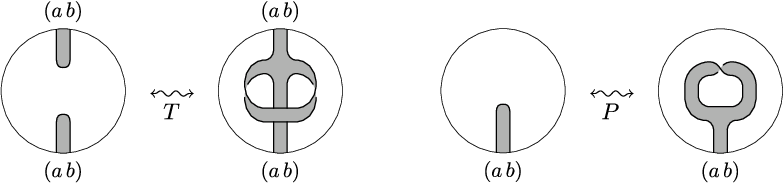}
    {}{Handle trading and blow-up moves for labeled ribbon surfaces.}
\end{Figure}

\begin{theorem}[\ns\cite{BP05,BP12}]\label{equiv-boundary/thm}
Let $L$ and $L'$ be two labeled links in $S^3$ representing closed connected oriented $3$-manifolds as ribbon fillable simple branched coverings of $S^3$ of the same degree $\geq 4$. Then, $L$ and $L'$ are related by labeled isotopy and the moves $B_1$ and $B_2$ in Figure \ref{bmoves/fig} if and only if the oriented $3$-manifolds they represent are PL homeomorphic.
\end{theorem}

\begin{Figure}[htb]{bmoves/fig}
\fig{}{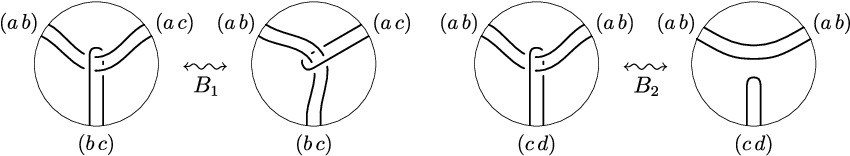}
    {}{Covering moves for labeled links bounding labeled ribbon surfaces.}
\end{Figure}

Now, before proceeding with the proofs of the theorems stated in Section \ref{stats/sec}, let us prove two lemmas.

\begin{lemma}\label{cobordism/thm}
Let $W \cong M \times [0,1] \cup H^1_1 \cup \dots \cup H^1_m \cup H^2_1 \cup \dots \cup H^2_n$ be an oriented $4$-dimensional $2$-cobordism between closed connected oriented $3$-manifolds\break $M_0 = M \times \{0\}$ and $M_1 = \Bd W - M_0$. Then, any $d$-fold simple branched covering $p_0\: M_0 \to S^3 \times \{0\}$ of degree $d \geq 3$ extends to a $d$-fold simple branched covering $p \: W \to S^3 \times [0,1]$, such that $B_p \subset S^3 \times [0,1]$ is a properly embedded locally flat PL surface. Moreover, if $p_0$ is ribbon fillable, then we can choose $p$ in such a way that also $p_1 = p_{|M_1}\:M_1 \to S^3 \times \{1\}$ is ribbon fillable.
\end{lemma}

\begin{proof}
This immediately follows from the main result in Edmonds \cite{Ed78} and its proof.
\end{proof}

\begin{remark}
If one is only interested in the existence of a 3-fold covering $p\:W \to S^3 \times [0,1]$ as in the lemma above, without insisting that it restricts to a given covering $p_0$, then the following argument provides a more explicit construction.

Let $K_0$ be a Kirby diagram representing a 4-dimensional 2-handlebody $W_0 = H^0 \cup H^2_{n+1} \cup \dots \cup H^2_\ell$ such that $\Bd W_0 \cong M$. By identifying a collar $C$ of $\Bd W_0$ in $W_0$ with $M \times [0,1] \subset W$, in such a way that $\Bd W_0$ corresponds to $M \times \{1\}$, we get a 4-dimensional 2-handlebody $W_1 = W_0 \cup_{C \cong M \times [0,1]} W = H^0 \cup H^1_1 \cup \dots \cup H^1_m \cup H^2_1 \cup \dots \cup H^2_n \cup H^2_{n+1} \cup \dots \cup H^2_\ell$. Here, the handles have been reordered in the usual way, once the attaching maps of the handles of $W$ are isotoped in $\Bd W_0$ out of the 2-handles of $W_0$. So, we have a Kirby diagram $K_1$ of $W_1$ that contains $K_0$ as a framed sublink.

Procedure \ref{SK/proc} determines a labeled ribbon surface $S_{K_1}$. By pushing the part of $S_{K_1}$ corresponding to $K_0$ a little bit more inside the interior of 
$B^4$ than the rest of $S_{K_1}$, we can assume that for some $r < 1$ the intersection of $S_{K_1}$ with the\break 4-ball $B^4_r \subset \Int B^4$ of radius $r$ is a copy of $S_{K_0}$ in $B^4_r$. Then, the branched covering $q_1 \: W_1 \to B^4$ represented by $S_{K_1}$ restricts to two branched coverings $q_0 \: W_0 \to B^4_r$ and $q \: W \to B^4 - \Int B^4_r$. At this point, the desired 3-fold simple branched covering $p \: W \to S^3 \times [0,1]$ is just the composition of $q$ with the canonical identification\break $B^4 - \Int B^4_r \cong S^3 \times [0,1]$.
\end{remark}

\begin{lemma}\label{concordance/thm}
Let $M$ be a closed connected oriented $3$-manifold and assume\break $d \geq 4$. For any two $d$-fold ribbon fillable simple branched coverings $p_0,p_1 \: M \to S^3$, there is a $d$-fold simple branched covering $p \: M \times [0,1] \to S^3 \times [0,1]$ satisfying the following properties: 1) the restriction $p_{|M \times\{0\}} \: M \times \{0\} \to S^3 \times \{0\}$ coincides with $p_0 \times \id_{\{0\}}$;
2) the restriction $p_{|M \times\{1\}} \: M \times \{1\} \to S^3 \times \{1\}$ is equivalent to $p_1 \times \id_{\{1\}}$ up to PL homeomorphisms; 3) the branch set $B_p \subset S^3 \times [0,1]$ is a properly immersed locally flat PL surface, whose singularities (if any) consist of an even number of transversal double points. In addition, if $d \geq 5$ there is such a branched covering $p$ with $B_p$ a properly embedded surface.
\end{lemma}

\begin{proof}
By Theorem \ref{equiv-boundary/thm}, the labeled links $L_0$ and $L_1$ representing the coverings $p_0$ and $p_1$, respectively, are related by labeled isotopy and moves $B_1$ and $B_2$ depicted in Figure \ref{bmoves/fig}. Each move $B_i$ can be realized as a composition of two iterations of the same Montesinos move $M_i$ depicted in Figure \ref{cmoves/fig}, applied in opposite directions and in the alternative form of Figure \ref{cmoves-bis/fig} for $i=1$ (the two directions are equivalent for $i=2$). This is shown in Figure \ref{cusps/fig} for $B_1$, while it is trivial for $B_2$ (cf. \cite[page 5]{BP05}, or the proof of Theorem 6.2.3 in \cite{BP12}).

\begin{Figure}[htb]{cmoves/fig}
\vskip3pt
\fig{}{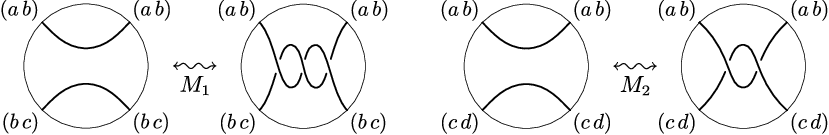}
    {}{Montesinos covering moves for labeled links.}
\end{Figure}

\begin{Figure}[htb]{cmoves-bis/fig}
\vskip3pt
\fig{}{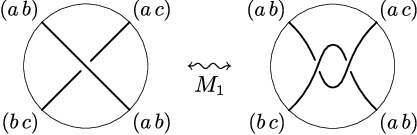}
    {}{Alternative form of move $M_1$.}
\end{Figure}

\begin{Figure}[htb]{cusps/fig}
\vskip6pt
\fig{}{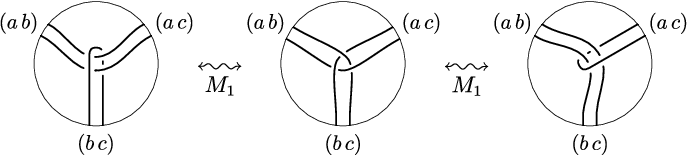}
    {}{Generating $B_1$ move as the composition of two (opposite) $M_1$ moves.}
\end{Figure}

The labeled links $L_0$ to $L_1$ can be joined by a family of singular links $L_t \subset S^3$ with $t \in [0,1]$, which present a singular point at a finite (even) number of values of $t$, say $t_1 < \dots < t_{2n}$, in correspondence of the Montesinos moves, while giving an isotopic deformation of (non-singular) links in each open interval $(t_i,t_{i+1})$ for $i = 1, \dots, 2n-1$. Following the argument proposed by Montesinos in \cite{Mo85}, and then used in Piergallini \cite{Pi95}, things can be arranged in such a way that $S = \cup_{t \in [0,1]} (L_t \times \{t\}) \subset S^3 \times [0,1]$ is a labeled locally flat PL surface with a cusp singularity (the cone of a trefoil knot) for each move $M_1$ and a node singularity\break (a transversal double point) for each move $M_2$. This is suggested by Figure \ref{singularities/fig}. 

\begin{Figure}[htb]{singularities/fig}
\fig{}{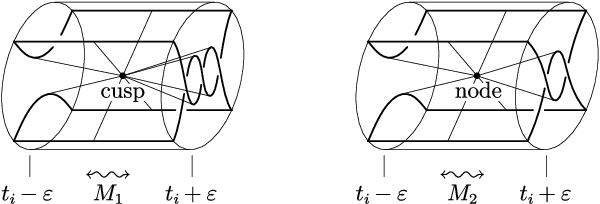}
    {}{Singularities of the branch surface deriving from Montesinos moves.}
\end{Figure}

Then, the labeled surface $S$ determines a $d$-fold simple branched covering $q \: M \times [0,1] \to S^3 \times [0,1]$, whose restrictions over $S^3 \times \{0\}$ and $S^3 \times \{1\}$ are equivalent to $p_0 \times \id_{\{0\}}$ and $p_1 \times \id_{\{1\}}$, respectively, up to PL homeomorphisms.
In particular, there exists a PL homeomorphisms $h\:M \to M$ such that $q_{|M \times \{0\}} \circ (h \times \id_{\{0\}}) = p_0 \times \id_{\{0\}}$, and we can replace $q$ by $q \circ (h \times \id_{[0,1]})$ to have the restriction over $S^3 \times \{0\}$ coinciding with $p_0 \times \id_{\{0\}}$ as required.

Here, cusp singularities come in pairs, each pair corresponding to two opposite moves $M_1$ and hence consisting of cones of a left-handed and a right-handed trefoil knot. Then, since $d \geq 4$, the technique described in \cite{Pi95} applies in the present context as well, being essentially local in nature, in order to remove all the (pairs of) cusp singularities (see Iori and Piergallini \cite{IP02} for a different approach). As the result we get a new labeled surface $S'$ representing a $d$-fold simple branched covering $p \: M \times [0,1] \to S^3 \times [0,1]$, such that $B_p = S'$ is a properly immersed locally flat PL surface whose singularities (if any) are transversal double points. Moreover, as shown in \cite{IP02}, if $d \geq 5$ transversal double points can also be removed in pairs from $B_p$ to give a properly embedded locally flat PL surface.
\end{proof}

At this point, we are ready to prove our main results. 

\begin{proof}[Theorem \ref{bc-cpt/thm}]
The existence of a branched covering $b$ as in the second part of the statement is guaranteed by Procedure \ref{SK/proc} applied to Kirby diagrams representing (4-dimensional 2-handlebodies bounded by) the components of $\Bd M$. So, we can directly assume that $b$ is given. We denote simply by $d=4$ or $5$, depending on the property \(a\) or \(b\) we desire, the degree $d(b)$ of this covering.

Let us start with the case $n = 1$, when $\Bd M$ is connected. Given any relative handlebody decomposition $H$ of $(M,\Bd M)$ with a single 4-handle and no 0-handles, let $M'$ consists of the 1-handles and the 2-handles of $H$ attached to a collar of $\Bd M$, and put $M'' = \Cl(M - M')$. Hence, $M'$ is an oriented 2-cobordism from $\Bd M$ to $\Bd M''$, while
we can think of $M''$ as a 4-dimensional 1-handlebody, by dualizing the 3-handles and the 4-handle of $H$. 

Lemma \ref{cobordism/thm} allows us to extend the given branched covering $b \: \Bd M \to \Bd B^4_1 \cong \Bd B^4_1 \times \{0\} \cong S^3 \times \{0\}$ to a $d$-fold simple covering $p' \: M' \to \Bd B^4_1  \times [0,1] \cong S^3 \times [0,1]$, such that the restriction $p'_1 = p'_| \: \Bd M'' \to S^3 \times \{1\} \cong S^3$ is ribbon fillable.
On the other hand, $M''$ is the boundary connected sum of a certain number $k$ of copies of $S^1 \times B^3$, hence it admits a standard representation as a 2-fold branched covering of $B^4$. This can be stabilized to a simple $d$-fold covering $p'' \: M'' \to B^4\label{cover-p'}$  branched over a ribbon surface (in fact, the union of $k+d-1$ separated trivial disks, with monodromies $(1\,2),\dots,(1\,2),(2\,3),(3\,4)$ and possibly $(4\,5)$, depending on $d$).

Now, Lemma \ref{concordance/thm} gives us a $d$-fold simple branched covering $q \: \Bd M'' \times [0,1] \to S^3 \times [0,1]$ satisfying \(a\) or \(b\) and such that, with the obvious canonical identifications, the restriction $q_0 = q_| \: \Bd M'' \times \{0\} \to S^3 \times \{0\}$  coincides with $p'_1$, while the restriction $q_1 = q_| \: \Bd M'' \times \{1\} \to S^3 \times \{1\}$ is equivalent to the restriction $p''_{\partial} = p''_| \: \Bd M'' \to S^3$ up to PL homeomorphisms. 

Then, we can glue together the coverings $p'$ and $p''$ through $q$, by identifying the corresponding restrictions, to obtain a $d$-fold simple branched covering $p \: M \to B^4$ with the property \(a\) or \(b\). In fact, according to Laudenbach and Po\'enaru \cite{LP72}, the result of the gluing is always PL homeomorphic to $M$, no matter what the homeomorphism occurring in the identification between $q_1$ and $p''_\partial$ is. This concludes the proof of the case $n = 1$.

The case $n > 1$ can be reduced to $n = 1$ as follows. Denote by $C_1, \dots, C_n$ the connected components of $\Bd M$. For every $i = 1, \dots, n$, we consider the restriction $b_i = b_{|}\:C_i \to \Bd B^4_i$ and a $d$-fold simple covering $q_i \: W_i \to B^4_i$ branched over a ribbon surface $B_{q_i} \subset B^4_i$ that extends $b_i$.
Then, we enlarge the 4-balls $B^4_1,\dots,B^4_n$ to disjoint PL 4-balls $\widehat B^4_1,\dots,\widehat B^4_n \subset S^4$ with a collar of their boundary, and each labeled surface $B_{q_i}$ to a properly embedded labeled ribbon surface $\widehat B_{q_i} \subset \widehat B^4_i$ by using the product structure along the collar.
Let $B^4 \subset S^4$ be a PL 4-ball obtained by attaching to $\widehat B^4_1 \cup \dots \cup \widehat B^4_n$ an embedded 1-handle between $\widehat B^4_i$ and $\widehat B^4_{i+1}$ for each $i=1,\dots,n-1$. These 1-handles can be chosen so that each attaching 3-ball meets $\Bd \widehat B_{q_1} \cup \dots \cup \Bd \widehat B_{q_n}$ in $d-1$ trivial arcs labeled $(1\,2),\dots,(d-1\,d)$. Finally, we attach labeled bands running along the connecting 1-handles of $B^4$ to get a labeled ribbon surface in $B^4$, as sketched in Figure \ref{csum/fig} (where the bands labeled $(4\,5)$ occur only if $d=5$).

\begin{Figure}[htb]{csum/fig}
\fig{}{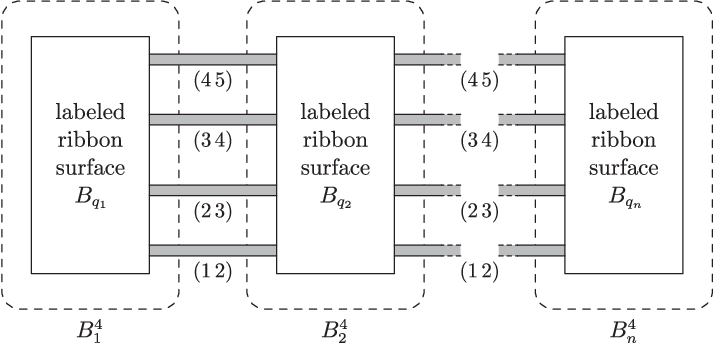}
    {}{The labeled ribbon surface $B_q$.}
\vskip-3pt
\end{Figure}

This is the labeled branch set $B_{q} \subset B^4$ of a $d$-fold simple branched covering $q\:W = W_0 \cup W_1 \cup \dots \cup W_n \to B^4$, where $W_0 \cong (\Bd M \times [0,1]) \cup H^1_1 \cup \dots \cup H^1_{n-1}$ is a 1-cobordism between $\Bd M = C_1 \cup \dots \cup C_n$ and $C \cong C_1 \cs \dots \cs C_n$, with the 1-handle $H^1_i$ connecting $C_i$ and $C_{i+1}$. 

Since $M$ is connected, we can assume $W_0 \subset M$ and put $M' = \Cl (M - W_0)$.
The restriction $q_0 = q_|\: W_0 \to B^4 - \Int(B^4_1 \cup \dots \cup B^4_n)$ is a simple branched covering whose branch surface is properly embedded in $B^4 - \Int(B^4_1 \cup \dots \cup B^4_n)$, while the restriction $q_{|}\:W_i \to B^4_i$ coincides with $q_i$ by construction for every $i=1,\dots,n$.

By construction, the restriction $q_{0|\Bd M'} \: \Bd M' \to S^3$ is ribbon fillable, bounding the covering $q$. Hence, by the case $n = 1$ proved above, we can extend such restriction to a simple covering $p' \: M' \to B^4$ satisfying property \(a\) or \(b\). Then, to obtain the wanted branched covering $p \: M \to S^4 - \Int (B^4_1 \cup \dots \cup B^4_n)$, we just glue $q_0$ and $p'$ together by identifying their restrictions over $S^3$.
\end{proof}

\begin{proof}[Theorem \ref{bc-cpt-bis/thm}]
Following the proof of Theorem \ref{bc-cpt/thm} for $n > 1$ and adopting the notations therein, we consider: the decomposition $\Bd M = C_1 \cup \dots \cup C_n$; the cobordism $W_0 = (\Bd M \times [0,1]) \cup H^1_1 \cup\dots \cup H^1_{n-1} \subset M$ between $\Bd M$ and $\Bd M' \cong C_1 \cs \dots \cs C_n$, with $M' = \Cl(M - W_0)$; the 3-fold simple branched coverings $q_i \: W_i \to B^4$ and their restrictions to the boundary $q_{i|} \: C_i \to S^3$, $i = 1, \dots, n$; the simple branched covering $p'\: M' \to B^4$.

Now, let $q'_0 \: \Bd M \times [0,1] \to S^3 \times [0,1]$ be the 3$n$-fold simple branched covering such that $C_i \times [0,1]$ is given by the three sheets from $3i-2$ to $3i$, and the restriction $q'_{0|C_i \times [0,1]}$ coincides with $q_{i|} \times \id_{[0,1]}$ up to a shifting by $3(i-1)$ in the numbering of the sheets. This covering $q'_0$ can be extended to a 3$n$-fold simple branched covering $q''_0 \: W_0 \to S^3 \times [0,1]$, by adding separated trivial disks $D_1, \dots, D_{n-1}$ to the labeled branch set, with $\Bd D_i \subset S^3 \times \{1\}$ and $D_i$ labeled by $(1\;3i + 1)$. The restriction of $q''_0$ over $S^3 \times \{1\}$ gives a 3$n$-fold simple branched covering $q''_{0|} \: \Bd M' \to S^3 \times \{1\}$.

Finally, we apply Theorem \ref{bc-cpt/thm} (actually the strong version of Lemma \ref{concordance/thm} provided by it, where the restriction $p_{|M \times \{1\}}$ coincides with the covering $p_{1} \times \id_{\{1\}}$) to connect such restriction with the restriction over the boundary of a stabilization to degree $3n > 5$ of the covering $p' \: M' \to B^4$. This gives the wanted simple branched covering $p \: M \to B^4$, and concludes the proof of the first part of the statement.

For the second part, it suffices to replace the covering $q_0'$ in the argument above by $b \times \id_{[0,1]}$, with $b \: \Bd M \to S^3$ any given 3$n$-fold ribbon fillable simple branched covering.
\end{proof}

\begin{proof}[Theorem \ref{bc-open/thm}]
By a standard argument, it is possible to construct an infinite family $\{M_i\}_{i \geq 0}$ of (non-empty) 4-dimensional compact connected PL submanifolds of $M$, such that $M = \cup_{i \geq 0}M_i$ and $M_i \subset \Int M_{i+1}$ for every $i \geq 0$. Then, we put $W_0 = M_0$ and $W_i = \Cl(M_i - M_{i-1})$ for every $i \geq 1$, and note that these are 4-dimensional compact PL submanifolds of $M$. Furthermore, for every $i \geq 1$, we can assume that each component $C$ of $W_i$ shares exactly one boundary component with $M_{i-1}$ (otherwise, if there are more shared components, we connect them by attaching to $M_{i-1}$ some 1-handles contained in $C \cap \Int M_i$). Let $\{C_v\}_{v \in V}$ be the set of all components of all the $W_i$'s, and $\{B_e\}_{e \in E}$ be the set of all their boundary components. We can think of $V$ and $E$ as the sets of vertices and edges of a graph $T$, respectively, with the edge $e \in E$ joining the vertices $v,w \in V$ if and only if $C_v$ and $C_w$ share the boundary component $B_e$. Actually, the above assumption about the intersection of the components of each $W_i$ and the corresponding $M_{i-1}$, implies that $T$ is a tree. We assume $T$ rooted at the vertex $v_0$ with $C_{v_0} = W_0$ and orient the edges of $T$ starting from $v_0$, so that each vertex $v \neq v_0$ has a single incoming edge $e_0^v$ and a non-empty set of outgoing edges $e^v_1, \dots, e^v_{n(v)}$ (we will use this notation also for the edges outgoing from $v_0$). According to this setting, the components of $W_i$ are the $C_v$ such that $d(v,v_0) = i$, where $d$ denotes the edge distance in $T$. Moreover, for each such component $C_v$ we have $\Bd C_v = B_{e^v_0} \cup B_{e^v_1} \cup \dots \cup B_{e^v_{\smash{n(v)}}}$, with $B_{e^v_0}$ the unique boundary component shared with $W_{i-1}$ if $i \geq 1$, while the boundary components $B_{e^v_1}, \dots , B_{e^v_{\smash{n(v)}}}$ are shared with $W_{i+1}$. In light of these facts, it is not difficult to see that $\End M \cong \End T$, with end points bijectively corresponding to infinite rays in $T$ starting from $v_0$.

Now, based on the same tree $T$, we want to construct a similar pattern in $S^4$
consisting of families ${\{C'_v\}}_{v \in V}$ and ${\{B'_e\}}_{e \in E}$. We begin with
any family ${\{B^4_e\}}_{e \in E}$ of standard PL 4-balls in $S^4$ satisfying the
following properties: 1) the diameter of $B^4_e$ vanishes when the edge distance
$d(e,v_0)$ goes to infinity; 2) $B^4_{e^v_1}, \dots, B^4_{e^v_{\smash{n(v)}}}$ are
pairwise disjoint for every $v \in V$ and contained in $\Int B^4_{e^v_0}$ if $v \neq v_0$.
Then, we put $C'_v = B^4_{e^v_0} - \Int(B^4_{e^v_1} \cup \dots \cup
B^4_{e^v_{\smash{n(v)}}})$ for every $v \in V$ (assume $B^4_{e^{\smash{v_0}}_0} = S^4$),
and $B'_e = \Bd B^4_e$ for every $e \in E$. By the very definition, we have $\Bd C'_v =
B'_{e^v_0} \cup B'_{e^v_1} \cup \dots \cup B'_{e^v_{\smash{n(v)}}}$ for every $v \in V$
(assume $B'_{e^{\smash{v_0}}_0} = \emptyset$). Moreover, $C'_v$ and $C'_w$ share the
boundary component $B'_e$ if and only if the edge $e$ joins the vertices $v,w \in V$ as
above, and thus $\End(\cup_{v \in V}C'_v) \cong \End T \cong \End M$.

The space $X = S^4 - \cup_{v \in V}C'_v = \cap_{i\geq 0}\cup_{d(e,v_0)=i} B^4_e$ is tame in $S^4$ (cf. Osborne \cite{Os66}). In particular, we can have $X \subset S^1$ by choosing each $B^4_e$ to be a round spherical 4-ball centered at a point of $S^1 \subset S^4$. We can conclude that $X \cong \End M$, being $S^4$ the Freudenthal compactification of $S^4 - X = \cup_{v \in V}C'_v$.

At this point, we can define the desired branched covering $p \: M \to S^4 - X$ in three steps. First, for every $e \in E$, we choose a Kirby diagram $K_e$ providing an integral surgery presentation of $B_e$, and denote by $p_e \: B_e \to B'_e \cong S^3$ the restriction to the boundary of the simple branched covering of $B^4$ determined by the labeled ribbon surface $S_{K_e}$, stabilized to degree 4 or 5, depending on the property \(a\) or \(b\) we want to obtain for $p$. Then, for every $v \in V$, we apply Theorem \ref{bc-cpt/thm} in order to extend $p_{e^v_0} \cup p_{e^v_1} \cup \dots \cup p_{e^v_{\smash{n(v)}}} \: \Bd B_v \to \Bd B'_v$ to a simple branched covering $p_v \: C_v \to C'_v$ satisfying property \(a\) or \(b\). Finally, we define $p = \cup_{v \in V}p_v \: M = \cup_{v \in V} C_v \to S^4 - X = \cup_{v \in V} C'_v$.
\end{proof}

\begin{proof}[Theorem \ref{bc-top/thm}]
Because of Theorem \ref{bc-clo/thm}, it suffices to consider the case when $M$ is not PL. Since any open 4-manifold admits a PL structure (see Lashof \cite{La71} or Freedman and Quinn \cite[Sec\-tion 8.2]{FQ90}), we can apply Theorem \ref{bc-open/thm} to the one-ended open connected oriented 4-manifold $M - \{x\}$, with $x$ any point of $M$, in order to get a PL branched covering $p \: M -\{x\} \to R^4$ satisfying property \(a\) or \(b\). Then, the one-point compactification of $p$ gives the wanted wild branched covering $q \: M \to S^4$, once $M$ and $S^4$ are identified with the one-point compactifications of $M - \{x\}$ and $R^4$, respectively.
\end{proof}

\begin{proof}[Theorem \ref{bc-open-bis/thm}]
Consider the decomposition $M = \cup_{i \geq 0} W_i\,$ and the families $\{C_v\}_{v \in V}$ and $\{B_e\}_{e \in E}$, as in the proof of Theorem \ref{bc-open/thm}. For every $i \geq 0$, put $M_i = W_i \cap W_{i+1} = \Bd W_i \cap \Bd W_{i+1}$ and observe that this is a closed 3-manifold with a finite number $n_i \leq n$ of components. The sequence $(n_i)_{i \geq 1}$ is non-decreasing, and without loss of generality we can assume $n_i \geq 2$ for every $i \geq 0$. Then, denoting by $S_i^3$ and $B_i^4$ respectively the 3-sphere and the 4-ball of radius $i$ in $R^4$, there is a 3$n_i$-fold\break simple branched covering $b_i \: M_i \to S_{i+1}^3$, bounding a 3$n_i$-fold simple covering of $B_{i+1}^4$ branched over a ribbon surface. Theorem \ref{cobordism-bis/thm}, which combines Theorems \ref{bc-cpt/thm} and \ref{bc-cpt-bis/thm} proved above, gives us 3$n_i$-fold simple coverings $p_i \: W_i \to \Cl(B_{i+1}^4 - B_i^4)$ with $i \geq 0$, such that the restrictions of $p_i$ over $S^3_i$ and $S^3_{i+1}$ respectively coincide with a 3$n_i$-fold stabilization of $b_{i-1}$ and with $b_i$ (where $b_{-1}$ is empty). At this point, we can glue the $p_i$'s together, up to stabilization. More precisely, we start with $p_0$, and then we add each $p_i$ in order, by gluing it to the appropriate $3n_i$-fold stabilization of $\cup_{j < i}p_j$. This gives the wanted 3$n$-fold branched covering $p = \cup_{i\geq0} p_i \: M = \cup_{i\geq0}W_i \to R^4$.
\end{proof}

\section{Final remarks}

We remark that all the simple branched coverings obtained in Theorems \ref{bc-cpt/thm}, \ref{bc-cpt-bis/thm}, \ref{bc-open/thm}, \ref{bc-open-bis/thm} and \ref{bc-top/thm} can be stabilized to any degree greater than the stated one. While this is obvious for branched coverings of $S^4$ or $B^4$ (like in Theorem \ref{bc-clo/thm}), in the other cases it can be achieved by suitable covering stabilizations in the construction process. 

We just sketch the case of Theorem \ref{bc-cpt/thm}, as the other cases can be treated in a similar way. In the proof of this theorem, it is enough to let $d$ be any given number $\geq 4$ from the beginning. Then, the monodromies of the $k + d - 1$ branch disks of the $d$-fold branched covering $p'' \: M'' \to B^4$ at page \pageref{cover-p'} change to $(1\,2), \dots, (1\,2), (2\,3), \dots, (d-1\,d)$, while the rest of the proof for the case $n = 1$ can be repeated word by word. Similarly, for the case $n > 1$ the only change is that the monodromies of the $d - 1$ bands in Figure \ref{csum/fig} become
$(1\,2), (2\,3), \dots, (d{-}1\,d)$.

We also remark that the arguments in the proofs of those theorems can be combined to prove various extensions of them. In particular, we have the following.
\begin{enumerate}
\item[1)] {\sl Any non-compact connected oriented PL 4-manifold $M$ whose boundary has only compact components, is a simple branched covering of $S^4 - (\End M \cup_{c \in C} \Int B^4_c)$, where ${\{B^4_c\}}_{c\in C}$ is a family of pairwise disjoint 4-balls in $S^4 - \End M$, indexed by the set $C$ of the boundary components of $M$. The limit set of the balls ${\{B^4_c\}}_{c\in C}$ is contained in $\End M \subset S^4$.}
\item[2)] {\sl Any compact connected oriented topological 4-manifold $M$ with boundary is a simple topological branched covering of $S^4 - \Int(\cup_{c \in C} B^4_c)$ with at most one wild point.}
\end{enumerate}

Finally, we observe that when $M$ does not admit a PL structure, the branch set $B_p$ of the covering $p \: M \to S^4$ in Theorem \ref{bc-top/thm} cannot be reduced to a locally flat PL surface properly immersed or embedded in $S^4$ by our proof. In fact, in this case there is a single wild point in $B_p$, at which we concentrate all the pathological aspects of the topology of $B_p$ and/or of the inclusion $B_p \subset S^4$. However, one might wonder if the situation could be simplified by diffusing the wild set $W_p$, or even more if such wild set could be eliminated at all, to get a tame topological branched covering at least under particular circumstances.

On the other hand, if the Kirby-Siebenmann invariant $\ks(M)$ is non-zero, then there is no a tame topological branched covering $p \: M \to S^4$ such that $B_p$ is an embedded or a self-transversally immersed topologically locally flat surface in $S^4$. Indeed, any such surface admits a compact tubular neighborhood $T \subset S^4$ by Freedman and Quinn \cite[Section 9.3]{FQ90}. Then, $T$ admits a PL structure such that $B_p$ is a PL embedded or immersed surface in $T$, hence $\ks(T) = 0$. Putting $U = \Cl(S^4-T) \subset S^4$, we also have $\ks(U) = 0$, since the Kirby-Siebenmann invariant is additive and $\ks(S^4) = 0$. It follows that $p^{-1}(T)$ is PL because it is a branched covering of $T$, and $p^{-1}(U) \times R$ is smoothable because it is an unbranched covering of $U \times R$, which we know to be smoothable. Therefore,  $\ks(M) = \ks(p^{-1}(T)) + \ks(p^{-1}(U)) = 0$.

So, we conclude with the following open problem.

\begin{question}
When, in representing a connected oriented topological 4-mani\-fold $M$ that is not PL by a simple branched covering $p \: M \to S^4$, can we require $B_p$ to be a topological surface wildly immersed or embedded in $S^4$? If $\ks(M) = 0$, can we require $p$ to be a tame topological branched covering, with $B_p$ a (locally) tame 2-complex or a topological surface (locally) tamely immersed or embedded in $S^4$?
\end{question}

\section*{Acknowledgements}
The authors are members of GNSAGA -- Istituto Nazionale di Alta Matematica ``Francesco Severi'', Italy.

The second author acknowledges support of the 2013 ERC Advanced Research Grant 340258 TADMICAMT.

The authors are grateful to the anonymous referee for his or her suggestions, which have been useful for improving the manuscript.



\thebibliography{00}

\bibitem{BP05} I. Bobtcheva and R. Piergallini, {\sl Covering moves and Kirby calculus}, preprint 2005, arXiv:math/0407032.

\bibitem{BP12} I. Bobtcheva and R. Piergallini, {\sl On 4-dimensional 2-handlebodies and  3-mani\-folds}, J. Knot Theory Ramifications {\bf 21} (2012), 1250110 (230 pages).

\bibitem{Ch78} P.T. Church, {\sl Discrete maps on manifolds}, Michigan Math. J. {\bf 25} (1978), 351--357.

\bibitem{Ed78} A.L. Edmonds, {\sl Extending a branched covering over a handle}, Pacific J. of Math. {\bf 79} (1978), 363--369.

\bibitem{Fo57} R.H. Fox, {\sl Covering spaces with singularities}, in ``Algebraic Geometry  and Topology. A symposium in honour of S. Lefschetz'', Princeton University Press 1957,  243--257.

\bibitem{Fr45} H. Freudenthal, {\sl \"Uber die Enden diskreter R\"aume und Gruppen}, Comment. Math. Helv. {\bf 17} (1945), 1--38.

\bibitem{FQ90} M.H. Freedman and F. Quinn, {\sl Topology of 4-manifolds}, Princeton  Math.\break Series {\bf 39}, Princeton University Press 1990.
		
\bibitem{IP02} M. Iori and R. Piergallini, {\sl 4-manifolds as covers of $S^4$ branched over non-singular surfaces}, Geometry \& Topology {\bf 6} (2002), 393--401.

\bibitem{Ki95} R. Kirby, {\sl Open problems in low-dimensional topology}, Geometric topology, Proceedings of the 1993 Georgia International Topology Conference, AMS/IP Studies in Advanced Mathematics, American Mathematical Society, 1997, 35--473. Available at math.berkeley.edu/\kern-1.5pt\raisebox{-10pt}{\LARGE\char126}kirby.

\bibitem{La71}
R. Lashof,
{\sl The immersion approach to triangulation and smoothing}, Algebraic topology (Proc. Sympos. Pure Math., Vol. XXII, Univ. Wisconsin, Madison, Wis., 1970), 131--164, Amer. Math. Soc., 1971.

\bibitem{LP72} F. Laudenbach and V. Po\'enaru, {\sl A note on 4-dimensional handlebodies}, Bull. Soc. Math. France {\bf 100} (1972), 337--344.

\bibitem{Mo78} J.M. Montesinos, {\sl 4-manifolds, 3-fold covering spaces and ribbons}, Trans. Amer. Math. Soc. {\bf 245} (1978), 453--467.

\bibitem{Mo85} J.M. Montesinos, {\sl A note on moves and irregular coverings of $S^4$}, Contemp. Math. {\bf 44} (1985), 345--349.

\bibitem{Mo02} J.M. Montesinos, {\sl Representing open 3-manifolds as 3-fold branched coverings}, Rev. Mat. Complut. {\bf 15} (2002), 533--542.

\bibitem{Mo05} J.M. Montesinos, {\sl Branched coverings after Fox}, Bol. Soc. Mat. Mexicana {\bf 11} (2005), 19--64.
 
\bibitem{Os66} R.P. Osborne, {\sl Embedding Cantor sets in a manifold}, Mich. Math. J. {\bf 13} (1966), 57--63.

\bibitem{Pi91} R. Piergallini, {\sl Covering Moves}, Trans Amer. Math. Soc. {\bf 325} (1991), 903--920.

\bibitem{Pi95} R. Piergallini, {\sl Four-manifolds as 4-fold branched covers of $S^4$}, Topology {\bf 34} (1995), 497--508.

\bibitem{PZ17-2} R. Piergallini and D. Zuddas, {\sl Branched coverings of $\CP^2$ and other basic 4-ma\-nifolds}, arXiv:1707.03667 (2017).

\bibitem{Po07} G. Pollini, {\sl Topological 4-manifolds as branched covers}, PhD thesis,  Universit\`a di Roma ``La Sapienza'', 2007.

\bibitem{RS72} C.P. Rourke and B.J. Sanderson, {\sl Introduction to piecewise-linear topology}, Ergebnisse der Mathematik und ihrer Grenzgebiete {\bf 69}, Springer-Verlag 1972.

\endthebibliography
\end{document}